\documentclass[11pt,letterpaper]{amsart}

\usepackage[utf8]{inputenc}
\usepackage{amssymb, amscd, amsmath, epsfig, mathtools}
\usepackage{amsthm}
\usepackage{enumerate}
\usepackage{xcolor}
\usepackage{scalerel}
\usepackage{soul}
\usepackage{tikz-cd}
\usepackage{esint}
\usepackage{cancel}
\usepackage{hyperref}
\usepackage{slashed}
\usepackage{url}

\usepackage[margin=1.0in]{geometry}

\newtheorem{theorem}{Theorem}[section]
\newtheorem*{theorem*}{Theorem}

\newtheorem{corollary}{Corollary}[section]
\newtheorem*{corollary*}{Corollary}
\newtheorem{lemma}{Lemma}[section]

\theoremstyle{definition}
\newtheorem{remark}{Remark}[section]
\newtheorem{example}{Example}[section]

\newcommand{\C}{\mathbb C}
\newcommand{\R}{\mathbb R}

\newcommand{\calC}{\mathcal C}

\newcommand{\calL}{\mathcal L}

\newcommand{\dvoll}{d\text{vol}_{\hat{g}}}

\begin{document}

\title[Positivity from $ X \times \C $ to $ X $]{On the Rosenberg-Stolz Conjecture for $ X \times \R^{2} $ and Its Application in Complex Geometry}
\author{Jie Xu}
\address{Department of Mathematics, Northeastern University, Boston, MA, USA}
\email{jie.xu@northeastern.edu}
%\date{April 2025}

\begin{abstract}
Let $ X $ be an oriented, closed manifold with $ \dim X \geqslant 2 $. In this article, we give both Riemannian geoemtry and complex geometry results on (sub)manifolds of the type $ X \times \C^{k} $ or $ X \times \R^{k} $. For Riemannian geometry side, we show that if $ X \times \C = X \times \R^{2} $ admits a Riemannian metric $ g $ with uniformly positive scalar curvature and bounded curvature, such that some novel conformally invariant $ g $-angle condition is satisfied, then there exists a metric $ \tilde{g} $ conformal to $ g $ such that $ \tilde{g} |_{X} $ has positive scalar curvature. This Riemannian path implies a complex geometry result: we show that if the complex manifold $ X \times \C $ admits a Hermitian metric $ \omega $ whose associated Riemannian metric $ g $ has uniformly positive scalar curvature and is of bounded curvature, then $ X \times \C $ admits a Hermitian metric $ \tilde{\omega} $ with positive Chern scalar curvature, provided that some $ g $-angle condition is satisfied. The Riemannian geometry result partially answers a 1994 Rosenberg-Stolz conjecture \cite{RosSto}. The complex geometry result extends a result of X.K. Yang \cite{Yang} from compact Hermitian manifolds to noncompact Hermitian manifolds of type $ X \times \C $. We further generalize both the Riemannian and complex geometry results to $ X \times \R^{k} $ or $ X \times \C^{k} $ for any $ k \geqslant 1 $ by imposing a generalized conformally invariant angle condition.
\end{abstract}
\maketitle

% This is the beginning of the introduction section.
\section{Introduction}
Let $ X $ be an oriented, closed manifold with $ \dim X \geqslant 2 $. In a 1994 survey article, Rosenberg and Stolz made the following conjecture \cite[Conjecture 7.1(2)]{RosSto}:
\medskip

{\it{Conjecture: If $ X $ admits no metrics of positive scalar curvature, then $ X \times \R^{2} $ admits no complete metrics of uniformly positive scalar curvature.}}
\medskip

This conjecture holds in many cases by imposing some topological significances, see e.g. \cite{JR}. Very recently, Cecchini, R\"ade and Zeidler \cite{CRZ} were able to verify this conjecture when $ X $ is oriented, $ 2\leqslant \dim_{\R}X \leqslant 5 $ and $ \dim_{\R}X \neq 4 $ by a $ \mu $-bubble method. In this article, we show by contrapositive statement that if for any closed, oriented manifold $ X $ with $ \dim_{\R}X \geqslant 2 $, the noncompact manifold $ X \times \R^{2} $ admits a complete metric $ g $ with uniformly positive scalar curvature, then $ X $ admits a positive scalar curvature in the same conformal class of the originally induced metric, provided that some novel conformally invariant angle condition and the geometry near ``infinity" are imposed. We also give an example, which states that if $ X $ has no PSC metric, and $ X \times \R^{2} $ admits a complete metric with positive scalar curvature and ``good" geometry near ``infinity", then the angle condition fails. It indicates that our angle condition is reasonable.
\medskip

Our conformal geometry approach allows us to apply the result of Rosenberg-Stolz conjecture for $ X \times \R^{2} $ to complex geometry. A 2020 result of X. K. Yang \cite[Corollary 3.9]{Yang} states that if $ (X, \omega, J) $ is a Hermitian manifold with Hermitian metric $ \omega $ and a fixed almost complex structure $ J $ whose background Riemannian metric $ g(J\cdot, \cdot) = \omega(\cdot, \cdot) $ has quasi-positive scalar curvature, then there exists a Hermitian metric $ \tilde{\omega} $ on $ X $ with positive Chern scalar curvature.

In this article, we also answer the following question, which is a generalization of X.K. Yang's results \cite{Yang}, \cite{Yang2} to noncompact Hermitian manifolds of cylindrical type $ X \times \C $:
\medskip

{\it{If $ (X \times \C, \omega, J) $ is a noncompact complex manifold with complete Hermitian metric $ \omega $ such that the background Riemannian metric  $ g(J\cdot, \cdot) = \omega(\cdot, \cdot) $ has uniformly positive scalar curvature, does $ X \times \C $ admits some complete Hermitian metric $ \tilde{\omega} $ with positive Chern scalar curvature?}}
\medskip

It is a natural first step from compact spaces to noncompact Hermitian manifold. Roughly speaking, (i) the existence of complete metrics with positive Riemannian/Chern scalar curvature on $ X \times \C $ is definitely related to the existence of positive Riemannian/Chern scalar curvature metrics on $ X $; (ii) $ X \times \C $ is a local model of every complex line bundle over the base space $ X $ (here $ X $ might be open subsets of the base manifold), which may build a pavement for the problem of transposing the positivity of Chern scalar curvatures from the total space of a complex line bundle to the base space.

The first reasoning (i) above can be observed from a Riemannian point of view. Topologically $ X \times \C \cong X \times \R^{2} $. It is clear that the Riemannian metric with uniformly positive scalar curvature on $ X \times \C $ is related to the positive Riemannian scalar curvature metric on $ X $, due to the 1994 Rosenberg-Stolz conjecture \cite[Conjecture 7.1]{RosSto}. This conjecture suggests a Riemannian path to solve the complex geometry problem: 

Step I: By assuming the uniform positivity of the Riemannian scalar curvature of the associated Riemannian metric $ g $ of $ \omega $, we address the following question: {\it{If $ X \times \R^{2} $ admits a Riemannian metric with uniformly positive scalar curvature, does $ X $ admit a Riemannian metric with positive Riemannian scalar curvature?}}

Step II: Once a Riemannian metric on $ X $ with positive Riemannian scalar curvature is found from Step I, we would like to know {\it{whether such a Riemannian metric is compatible with some Hermitian structure on $ X $?}}

Step III: If we get affirmative answers for both Steps I and II, we obtain a Hermitian metric $ \tilde{\omega} $ on $ X $ with positive Chern scalar curvature by X. K. Yang \cite{Yang}, \cite{Yang2}. Therefore we get a Hermitian metric with positive Chern scalar curvature via the product metric of $ \tilde{\omega} $ and standard Euclidean metric of $ \C $.

According to our Riemannian path, there are topological obstructions given in e.g. \cite{CRZ}, \cite{HPS}, \cite{JR}, \cite{RosSto}, \cite{Zeilder2}, therefore we cannot expect an affirmative answer without any further restriction. However, instead of topological restrictions, we impose a conformally invariant geometric condition, which applies to all dimensions. In addition, we need to impose geometric conditions near ``infinity". For any complete Riemannian manifold $ (N, \hat{g}) $, we say that $ \hat{g} $ is of {\it{bounded curvature}} in the sense of Aubin \cite[Chapter 2]{Aubin} if $ \hat{g} $ is a complete Riemannian metric, having positive injectivity radius, and having uniformly bounded sectional curvature.

To state the main results, we need some notation. Our motivation for this analytic approach is given in \S2. Let $ n - 1 : = \dim_{\R} X $. Denote by $ X \times \C \cong X \times \R^{2} \cong X \times \R_{\xi} \times \R_{\zeta} $, i.e. the two real lines are assigned globally $ \xi $- and $ \zeta $-variables, respectively. It follows that we have global vector fields $ \partial_{\xi} \in \Gamma(T\R_{\xi}) $ and $ \partial_{\zeta} \in \Gamma (T\R_{\zeta}) $, respectively. Let $ g $ be the Riemannian metric on $ X \times \R^{2} $. Denote the projection maps and associated induced metrics by
\begin{equation}\label{intro:eqn1}
\begin{tikzcd}
    (X \times \R_{\xi} \times \R_{\zeta}, g) \arrow[r, shift left = 1.5ex, "\pi_{\xi}"] & (X \times \R_{\zeta}, \imath_{\xi}^{*}g) \arrow[r, shift left = 1.5ex, "\pi_{\zeta}"] \arrow[l, "\imath_{\xi}"] & (X, \imath_{\zeta}^{*} \imath_{\xi}^{*}g) \arrow[l, "\imath_{\zeta}"].
\end{tikzcd}
\end{equation}
By (\ref{intro:eqn1}), we identify $ X \times \R_{\zeta} $ with $ X_{\xi, 0} : = X \times R_{\zeta} \times \lbrace 0 \rbrace_{\xi} $, and $ X $ with $ X_{\zeta, 0} : = X \times \lbrace 0 \rbrace_{\zeta} = X \times \lbrace 0 \rbrace_{\zeta} \times \lbrace 0 \rbrace_{\xi} $; we also denote by $ \nu_{g} $ and $ \nu_{\imath_{\xi}^{*}g} $ the unit normal vector field along $ X_{\xi, 0} $ and $ X_{\zeta, 0} $, respectively. In particular, we choose the direction such that $ g(\nu_{g}, \partial_{\xi}) \geqslant 0 $ and $ \imath_{\xi}^{*}g(\nu_{\imath_{\xi}^{*}g}, \partial_{\zeta}) \geqslant 0 $, respectively. Let $ R_{g}, R_{\imath_{\xi}^{*}g}, R_{\imath_{\zeta}^{*}\imath_{\xi}^{*}g} $ be the Riemannian scalar curvatures on $ X \times \R^{2}, X_{\xi, 0}, X_{\zeta, 0} $ respectively. Let $ \text{Ric}_{g} , \text{Ric}_{\imath_{\xi}^{*}g} $ be the Riemannian Ricci curvature tensors on $ X \times \R^{2}, X_{\xi, 0} $, respectively. Let $ A_{g}, h_{g}, A_{\imath_{\xi}^{*}g}, h_{\imath_{\xi}^{*}g} $ be the Riemannian second fundamental forms and Riemannian mean curvatures on $ X_{\xi, 0}, X_{\zeta, 0} $, respectively. 

Let $ (d\pi_{\xi})_{*} \nu_{g} $ be the projection of $ \nu_{g} \in \Gamma(T(X \times \R^{2})) $ onto $ X_{\xi, 0} $ via $ \pi_{\xi} $. Note that $ (d\pi_{\xi})_{*} \nu_{g} \in \Gamma(T(X \times \R_{\zeta})) $ is also a unit vector field. We define the $ \imath_{\xi}^{*}g $-angle between $ (d\pi_{\xi})_{*} \nu_{g} $ and $ \partial_{\zeta} $ along $ X_{\zeta, 0} $ by
\begin{equation*}
    \cos (\angle_{\imath_{\xi}^{*}g}(d\pi_{\xi})_{*} \nu_{g}, \partial_{\zeta})) : = \frac{\imath_{\xi}^{*}g((d\pi_{\xi})_{*} \nu_{g}, \partial_{\zeta})}{\imath_{\xi}^{*}g((d\pi_{\xi})_{*} \nu_{g}, (d\pi_{\xi})_{*} \nu_{g})^{\frac{1}{2}} \imath_{\xi}^{*}g(\partial_{\zeta}, \partial_{\zeta})^{\frac{1}{2}}} = \frac{\imath_{\xi}^{*}g((d\pi_{\xi})_{*} \nu_{g}, \partial_{\zeta})}{ \imath_{\xi}^{*}g(\partial_{\zeta}, \partial_{\zeta})^{\frac{1}{2}}} \in \left[0, \frac{\pi}{2} \right].
\end{equation*}
Similarly, we define the $ \imath_{\xi}^{*}g $-angle between $ \nu_{\imath_{\xi}^{*}g} $ and $ \partial_{\zeta} $ by
\begin{equation*}
    \cos (\angle_{\imath_{\xi}^{*}g}(\nu_{\imath_{\xi}^{*}g}, \partial_{\zeta})) : = \frac{\imath_{\xi}^{*}g(\nu_{\imath_{\xi}^{*}g}, \partial_{\zeta})}{\imath_{\xi}^{*}g(\nu_{\imath_{\xi}^{*}g}, \nu_{\imath_{\xi}^{*}g})^{\frac{1}{2}} \imath_{\xi}^{*}g(\partial_{\zeta}, \partial_{\zeta})^{\frac{1}{2}}} = \frac{\imath_{\xi}^{*}g(\nu_{\imath_{\xi}^{*}g}, \partial_{\zeta})}{ \imath_{\xi}^{*}g(\partial_{\zeta}, \partial_{\zeta})^{\frac{1}{2}}} \in \left[0, \frac{\pi}{2} \right].
\end{equation*}
Clearly both quantities are conformally invariant with respect to any conformal transformation in $ X \times \R^{2} $. Our main result in Riemannian geometry is:
\begin{theorem}\label{intro:thm1}
Let $ X $ be an closed, oriented Riemannian manifold with $ \dim_{\R}X \geqslant 2 $. Assume that $ (X \times \R^{2}, g) $ admits a complete Riemannian metric $ g $ that is of bounded curvature, and such that $ R_{g} \geqslant \kappa_{0} > 0 $ for some $ \kappa_{0} > 0 $. Assume
\begin{equation}\label{intro:eqn1a}
\begin{split}
 & \frac{n - 1}{n - 2} \sec^{2} (\angle_{\imath_{\xi}^{*}g}(d\pi_{\xi})_{*} \nu_{g}, \partial_{\zeta})) + \sec^{2} (\angle_{\imath_{\xi}^{*}g}(\nu_{\imath_{\xi}^{*}g}, \partial_{\zeta}))  \\
 & \qquad = \frac{n - 1}{n - 2} \frac{\imath_{\xi}^{*}g(\partial_{\zeta},\partial_{\zeta})}{\imath_{\xi}^{*}g((d\pi_{\xi})_{*}\nu_{g},\partial_{\zeta})^{2}} + \frac{\imath_{\xi}^{*}g(\partial_{\zeta},\partial_{\zeta})}{\imath_{\xi}^{*}g(\nu_{\imath_{\xi}^{*}g},\partial_{\zeta})^{2}}<
 2 + \frac{n - 1}{n - 2}
\end{split}
\end{equation}
on $ X \times \lbrace 0 \rbrace_{\xi} \times \lbrace 0 \rbrace_{\zeta} $ if $ (d\pi_{\xi})_{*}\nu_{g} $ is nowhere vanishing along $ X \times \lbrace 0 \rbrace_{\xi} \times \lbrace 0 \rbrace_{\zeta} $; otherwise assume
\begin{equation}\label{intro:eqn1b}
\sec^{2} (\angle_{\imath_{\xi}^{*}g}(\nu_{\imath_{\xi}^{*}g}, \partial_{\zeta}))   = \frac{\imath_{\xi}^{*}g(\partial_{\zeta},\partial_{\zeta})}{\imath_{\xi}^{*}g(\nu_{\imath_{\xi}^{*}g},\partial_{\zeta})^{2}} < 2 
\end{equation}
when $ (d\pi_{\xi})_{*}\nu_{g} \equiv 0 $ on $ X \times \lbrace 0 \rbrace_{\xi} \times \lbrace 0 \rbrace_{\zeta} $. Then there exists a Riemannian metric $ \tilde{g} $ in the conformal class of $ g $ such that $ R_{\imath_{\zeta}^{*}\imath_{\xi}^{2}\tilde{g}} > 0 $ on $ X $.
\end{theorem}
\begin{remark}\label{intro:re1} 
Note that the same conclusion follows if we replace the hyperplane by any $ X \times \lbrace \zeta \rbrace \times \lbrace \xi \rbrace $ and apply the associated extrinsic geometry and angle condition. Therefore we only require the angle condition to be held along some hypersurface $ X \times \lbrace \zeta \rbrace \times \lbrace \xi \rbrace $. It follows that Theorem \ref{intro:thm1} partially answers the 1994 Rosenberg-Stolz conjecture \cite[Conjecture 7.1(2)]{RosSto}, as we stated in Corollary \ref{Riem:cor1}.  In particular, the resulting Riemannian metric with positive Riemannian scalar curvature is within the same conformal class of the original Riemannian metric.

We point out that this type of transpose of some geometric notions of positivity relies on a new conformally invariant quantity--the angle condition. The angle condition is an extrinsic geometric condition along hypersurfaces, and can be viewed as a substitute for a mean curvature condition. In fact, positive Riemannian scalar curvature in the ambient space plus some mean curvature condition along the submanifolds is not enough to guarantee positive scalar curvature on submanifolds.

Clearly the product metric $ g_{X} \oplus d\zeta^{2} \oplus d\xi^{2} $ with uniform PSC is a trivial example which satisfies the angle condition. In Example \ref{Mot:ex1}, we set $ X = \mathbb{T}^{n - 1}, n \geqslant 3 $, and construct a complete metric with positive scalar curvature on $ X \times \R^{2} $, such that the angle condition fails on $ X \times \lbrace \zeta \rbrace \times \lbrace \xi \rbrace, \forall \zeta \in \R_{\zeta}, \xi \in \R_{\xi} $. We also give further discussions in \S2, especially in Remark \ref{Mot:re1}.
\end{remark}
Taking the Riemannian path, the result of Theorem \ref{intro:thm1} implies our main result in complex geometry on $ X \times \C $:
\begin{theorem}\label{intro:thm2}
Let $ X $ be a complex manifold with $ \dim_{\C}X \geqslant 1 $. Assume that $ (X \times \C, J, \omega) $ admits a complete Hermitian metric $ \omega $ whose background metric $ g $ is of bounded curvature, and such that $ R_{g} \geqslant \kappa_{0} > 0 $ for some $ \kappa_{0} > 0 $. If in addition (\ref{intro:eqn1a}) or (\ref{intro:eqn1b}) holds on $ X \times \lbrace 0 \rbrace_{\xi} \times \lbrace 0 \rbrace_{\zeta} $, then there exists a Hermitian metric $ \tilde{\omega} $ with positive Chern scalar curvature on $ X \times \C $.
\end{theorem}
Theorem \ref{intro:thm2} generalizes the result of X. K. Yang \cite{Yang}, \cite{Yang2} from closed Hermitian manifolds to noncompact manifolds of cylindrical type. In particular, $ X \times \C $ may not be a trivial holomorphic line bundle, but just a trivial real vector bundle of rank $ 2 $.
\medskip

The Rosenberg-Stolz conjecture, with the direction from non-positivity of $ X $ to the non-positivity of $ X \times \R^{k} $, cannot be strengthened when $ k \geqslant 3 $, due to \cite[Proposition 7.2]{RosSto}: Given any $ X $, $ X \times \R^{k} $ always admits a complete Riemannian metric with uniformly positive Riemannian scalar curvature. But from analytic point of view, we can transpose the positivity from $ X \times \R^{k} $ to $ X $, for all $ k \in \mathbb{N} $, with an appropriate angle condition analogous to (\ref{intro:eqn1a}). In other words, if $ X \times \R^{k} $ admits some notions of positivity, and the geometry of $ X \times \R^{k} $ is not ``too bad", then $ X $ must ``inherit" some notions of positivity.

Denote by $ X \times \R^{k} : = X \times \R_{1} \times \R_{2} \times \dotso \times \R_{k} $ with $ \R_{i} \cong \R, \forall i $. For general $ j $, we also denote by
\begin{equation*}
\begin{tikzcd}
    (X \times \R_{1} \times \dotso \times \R_{j}, \imath_{j + 1}^{*} \dotso \imath_{k}^{*} g) \arrow[r, shift left = 1.5ex, "\pi_{j}"] & (X \times \R_{1} \times \dotso \times \R_{j - 1}, \imath_{j}^{*} \dotso \imath_{k}^{*}g) \arrow[l, "\imath_{j}"] 
\end{tikzcd}
\end{equation*}
We denote by $ \nu_{j + 1} \in \Gamma(X \times \R_{1} \times \dotso \times R_{j + 1}) $ the unit normal vector field along $ X \times \R_{1} \times \dotso \times \R_{j} $, and the canonical vector field by $ \partial_{j + 1} \in \R_{j + 1} $ for $ j = 0, \dotso, k - 1 $. Denote by $ \mu_{j} = d(\pi_{2} \circ \dotso \circ \pi_{j})_{*}\nu_{j} \in \Gamma(X \times \R_{1}), j = 2, \dotso, k $, and the induced metric on $ X \times \R_{1} \times \dotso \times \R_{j - 1} $ by $ \imath_{j}^{*} \dotso \imath_{k}^{*}g : = g_{j - 1}, j = 2, \dotso, k $. Set
\begin{equation}\label{intro:eqn4a}
    A_{j} = \frac{n + j - 3}{n - 2} \sec^{2}\left(\angle_{g_{1}}(\mu_{j}, \partial_{1}) \right) = \frac{n + j - 3}{n - 2} \frac{g_{1}(\partial_{1}, \partial_{1})}{g_{1}(\mu_{j}, \partial_{1})^{2}}, B_{j} = \frac{n + j - 3}{n - 2}, j = 2, \dotso, k.
\end{equation}
For instance, $ \pi_{2} = \pi_{\xi} $, $ \nu_{2} = \nu_{g} $ and $ \mu_{2} = d(\pi_{\xi})_{*}\nu_{g} $ by our notation of (\ref{intro:eqn1}) when $ n = 2 $. We now generalize Theorem \ref{intro:thm1} to the result of transposing the positivity of Riemannian scalar curvature from $ X \times \R^{k} $ to $ X $. \begin{theorem}\label{intro:thm3}
Let $ X $ be an oriented, closed manifold with $ \dim_{\R}X \geqslant 2 $. Assume that $ X \times \R^{k} $ admits a Riemannian metric $ g $ that is of bounded curvature, and such that $ R_{g} \geqslant \kappa_{0} > 0 $ for some $ \kappa_{0} > 0 $. If
\begin{equation}\label{intro:eqn4}
\sum_{j = 2}^{k} A_{j} + \sec^{2}(\angle_{g_{1}}(\nu_{1}, \partial_{1})) = \sum_{j = 2}^{n} A_{j} + \frac{g_{1}(\partial_{1}, \partial_{1})}{g_{1}(\nu_{1}, \partial_{1})^{2}} < 2 + \sum_{j = 2}^{k} B_{j},
\end{equation}
on $ X \times \lbrace 0 \rbrace_{1} \times \dotso \times \lbrace 0 \rbrace_{k} $ (where, for $ j = 2, \dotso, n $, we set $ A_{j} \equiv 0 $ and remove associated $ B_{j} $ on the right hand side of (\ref{intro:eqn4}) whenever $ \mu_{j} \equiv 0 $ along $ X \times \R_{1} $), then there exists a Riemannian metric $ \tilde{g} $ in the conformal class of $ g $ such that $ R_{\imath_{1}^{*} \dotso \imath_{k}^{*} \tilde{g}} > 0 $ on $ X $.
\end{theorem}
When $ k = 2k' $ is an even number, and $ X $ is a complex manifold, we generalize the result of Theorem \ref{intro:thm2} on $ X \times \C^{k} $, according to the same Riemannian path.
\begin{theorem}\label{intro:thm4}
Let $ X $ be a complex manifold with $ \dim_{\C}X \geqslant 1 $. Assume that $ (X \times \C^{k}, J, \omega) $ admits a complete Hermitian metric $ \omega $ whose background metric $ g $ satisfies the same hypotheses in Theorem \ref{intro:thm3}, then there exists a Hermitian metric $ \tilde{\omega} $ with positive Chern scalar curvature on $ X \times \C^{k} $.
\end{theorem}
\medskip

This article is organized as follows: In \S2, we introduce our motivation from a conformal geometry point of view. We also introduce related works by different approaches. In particular, we give examples and discussions to show the necessity of this angle condition to some extent.

In \S3, \S4, and \S5, we address the problems with ambient space either $ X \times \R^{2} $ or $ X \times \C $.  In \S3, we construct a geometric partial differential equation in (\ref{Pre:eqn5}). We confirm the ellipticity of the differential operator in Lemma \ref{Pre:lemma1}. We construct the inhomogeneous term of the partial differential equation in Lemma \ref{Pre:lemma2}. With the desired inhomogeneous term and the ellipticity, we obtain $ \calC^{1, \alpha} $-smallness of the solution in Lemma \ref{Pre:lemma3}, and give partial $ \calC^{2} $-estimates in Lemma \ref{Pre:lemma5} with the help of the global positivity established in Lemma \ref{Pre:lemma4}. After defining conformal factors in various spaces, we compare the Laplacians for various spaces in Lemma \ref{Pre:lemma6} with respect to functions in (\ref{Pre:eqn7}).

In \S4, we apply the preliminaries in \S3 to prove in Theorem \ref{Riem:thm1} that there exists a Riemannian metric with positive Riemannian scalar curvature on $ X $ by imposing the angle condition, and such a metric is conformal to the original metric. As a contrapositive statement, we partially answer the 1994 Rosenberg-Stolz conjecture in Corollary \ref{Riem:cor1}. In \S5, we apply Theorem \ref{Riem:thm1} and the feature of conformal transformation to generalize the result of X. K. Yang \cite[Corollary 3.9]{Yang} in Theorem \ref{Ch:thm2} for noncompact complex manifolds of the type $ X \times \C $, and therefore complete our Riemannian path to this complex geometry problem.

In \S6, we further generalize our results to either Riemannian manifolds $ X \times \R^{k} $ or complex manifolds $ X \times \C^{k} $, for any $ k \in \mathbb{N} $, respectively. The techniques used are the same as in previous sections, with just some added notational complexity.
% We show in Theorem \ref{Gen:thm1} that the positivity of Riemannian scalar curvature on $ X \times \R^{k} $ induces a positive Riemannian scalar curvature metric within the same conformal class of the original metric, provided that a generalized angle condition must hold. As a corollary, we show that the existence of Hermitian metric with uniformly positive Riemannian scalar curvature on $ X \times \C^{k} $ implies the existence of a Hermitian metric with positive Chern scalar curvature in Corollary \ref{Gen:cor1}, when the conformally invariant generalize angle condition is assumed. Besides the adjustment of the angle condition and the associated differential operator, the rest proof of the general cases is exactly the same as the proof for $ X \times \R^{2} \cong X \times \C $ when compare the diagrams (\ref{Pre:diagram}) and (\ref{Gen:diagram}), it is just notationally more difficult for general $ k $.
\medskip

% Motivations.
\section{Related Work, Motivation and Examples of Angle Condition}
As we mentioned in \S1, we resolve the complex geometry result via a Riemannian path. We consider this type of problems from a conformal geometry point of view, along with repeated Gauss-Codazzi equations on codimension one hypersurfaces. Such a Riemannian path unveils the positivity of the Riemannian scalar curvature in the submanifold $ X $ from the positivity of the Riemannian scalar curvature in the ambient space $ X \times \R^{2} $. 

From a contrapositive statement point of view, the 1994 Rosenberg-Stolz conjecture is the type of problem by asking the existence of positive scalar curvature on submanifolds, provided that some notions of positivity are given in the ambient spaces. For noncompact manifolds, Step I results hold for many cases. Generally, there are two types of methods in the study of scalar curvature problems on hypersurfaces: ones is based on Dirac operator and index theory on spin manifolds; the other is based on the geometric measure theory initiated by Schoen-Yau minimal hypersurfaces method \cite{SY2}; as a modified version of the minimal surface method, the $ \mu $-bubble method, pioneered by Gromov's seminal work \cite{GROMOV}, \cite{GROMOV2}, is another major stream in the study of positive scalar curvature manifolds. A comprehensive survey is given by \cite{JR}. In particular, the studies via index-theoretic obstructions can be found in \cite{HPS}, \cite{JR}, \cite{RosSto}, \cite{Zeilder2}; the studies via $ \mu $-bubbles or minimal surface method can be found in \cite{CRZ}. However, they provide no information whether their Riemannian metrics with positive Riemannian scalar curvature are compatible with Hermitian structures. 

We solve this type of problem according to a different approach. By introducing an auxiliary $ 1 $-dimensional space, and by applying geometric partial differential equations that are coming from the conformal transformation of the geometric quantities in the Gauss-Codazzi equation, we successfully answered these type of questions on closed manifolds \cite{RX2}($ \mathbb{S}^{1} $-stability conjecture \cite[Conjecture 1.24]{JR}), on noncompact manifolds with boundary \cite{XU10} (generalization of Gromov-Lawson scalar and curvature comparison theorem on compact cylinders \cite{GL}), and on noncompact manifolds $ X \times \R $ \cite{XU11}(Rosenberg-Stolz \cite[Conjecture 7.1(1)]{RosSto}), provided that some geometric angle conditions are assumed. Since conformal transformation preserves the Hermitian structure, it is natural to apply our conformal geometry and partial differential equation method to address the scenario $ X \times \C $ with the introduction of some geometric angle condition and an appropriate auxiliary space. 

In a word, our conformal geometry and geometric analysis approach applies equally well in transposing some notions of positivity from ambient Riemannian/Hermitian manifolds to their (sub)manifolds.

Let's consider the Step I in \S1 first. By Riemannian Gauss-Codazzi equation on $ X_{\xi, 0} $,
\begin{equation*}
    R_{g} = R_{\imath_{\xi}^{*}g} + 2\text{Ric}_{g}(\nu_{g}, \nu_{g}) + \lVert A_{g} \rVert^{2} - h_{g}^{2} \: {\rm on} \; X_{\xi, 0}.
\end{equation*}
Here the scalar curvature $ R_{g} $ is really $ R_{g} = R_{g} |_{X_{\xi, 0}} = \imath_{\xi}^{*} R_{g} $. And the same manner follows for other functions. Applying Gauss-Codazzi equation again on $ X_{\zeta, 0} $,
\begin{equation*}
    R_{\imath_{\xi}^{*}g} = R_{\imath_{\zeta}^{*} \imath_{\xi}^{*}g} + 2\text{Ric}_{\imath_{\xi}^{*}g}(\nu_{\imath_{\xi}^{*}g}, \nu_{\imath_{\xi}^{*}g}) + \lVert A_{\imath_{\xi}^{*}g} \rVert^{2} - h_{\imath_{\xi}^{*}g}^{2} \; {\rm on} \; X_{\zeta, 0}.
\end{equation*}
Combining the above two equations together, we have
\begin{equation}\label{intro:eqn2}
R_{\imath_{\zeta}^{*} \imath_{\xi}^{*}g} = R_{g} - 2\text{Ric}_{\imath_{\xi}^{*}g}(\nu_{\imath_{\xi}^{*}g}, \nu_{\imath_{\xi}^{*}g}) - 2\text{Ric}_{g}(\nu_{g}, \nu_{g}) + h_{g}^{2} + h_{\imath_{\xi}^{*}g}^{2} - \lVert A_{g} \rVert^{2} - \lVert A_{\imath_{\xi}^{*}g} \rVert^{2} \; {\rm on} \; X_{\zeta, 0},
\end{equation}
that holds for every Riemannian metric $ g $. Given $ R_{g} \geqslant \kappa_{0} > 0 $ for some $ \kappa_{0} \in \R_{+} $, we want $ R_{\imath_{\zeta}^{*} \imath_{\xi}^{*}g} > 0 $ by (\ref{intro:eqn2}). But knowing only the positivity of the scalar curvature is not enough to evaluate the sign of Ricci curvature tensors and the largeness of the second fundamental forms. 

The conformal class of $ g $, denoted by $ [g] $, is defined to be
\begin{equation*}
    [g] = \lbrace e^{2\phi}g : \phi \in \calC^{\infty}(X \times \R^{2}) \rbrace.
\end{equation*}
The conformal transformation of $ g $, say $ \tilde{g} = e^{2\phi} g $, induces conformal metrics $ \imath_{\xi}^{*} \tilde{g} = e^{2\imath_{\xi}^{*} \phi} \imath_{\xi}^{*}g $ and $ \imath_{\zeta}^{*}\imath_{\xi}^{*} \tilde{g} = e^{2\imath_{\zeta}^{*}\imath_{\xi}^{*} \phi} \imath_{\zeta}^{*}\imath_{\xi}^{*}g $ on $ X_{\xi, 0} $ and $ X_{\zeta, 0} $, respectively. Recall that $ \dim_{\R} X = n - 1, \dim_{\R} (X \times \R_{\zeta}) = n, \dim_{\R} (X \times \R^{2}) = n + 1 $. Applying (\ref{intro:eqn2}) with respect to $ \tilde{g} $ and associated unit normal vector fields $ \nu_{\tilde{g}} = e^{-\phi} \nu_{g} $ and $ \nu_{\imath_{\xi}^{*}\tilde{g}} = e^{-\imath_{\xi}^{*} \phi} \nu_{\imath_{\xi}^{*}g} $, also by conformal transformations of Ricci curvature tensors, scalar curvatures, second fundamental forms and mean curvatures, we have
\begin{equation}\label{intro:eqn3}
\begin{split}
R_{\imath_{\zeta}^{*} \imath_{\xi}^{*}\tilde{g}} & = \left( R_{\tilde{g}} - 2\text{Ric}_{\imath_{\xi}^{*}\tilde{g}}(\nu_{\imath_{\xi}^{*}\tilde{g}}, \nu_{\imath_{\xi}^{*}\tilde{g}}) - 2\text{Ric}_{\tilde{g}}(\nu_{\tilde{g}}, \nu_{\tilde{g}}) + h_{\tilde{g}}^{2} + h_{\imath_{\xi}^{*}\tilde{g}}^{2} - \lVert A_{\tilde{g}} \rVert^{2} - \lVert A_{\imath_{\xi}^{*}\tilde{g}} \rVert^{2} \right) \bigg|_{X_{\zeta, 0}} \\
& = e^{-2\imath_{\zeta}^{*}\imath_{\xi}^{*}\phi} \left( R_{g} - 2\text{Ric}_{\imath_{\xi}^{*}g}(\nu_{\imath_{\xi}^{*}g}, \nu_{\imath_{\xi}^{*}g}) - 2\text{Ric}_{g}(\nu_{g}, \nu_{g}) + h_{g}^{2} + h_{\imath_{\xi}^{*}g}^{2} - \lVert A_{g} \rVert^{2} - \lVert A_{\imath_{\xi}^{*}g} \rVert^{2} \right) \bigg|_{X_{\zeta, 0}} \\
& \qquad + e^{-2\imath_{\zeta}^{*}\imath_{\xi}^{*}\phi} \left(2(n - 1)\nabla_{\nu_{g}} \nabla_{\nu_{g}} \phi + 2(n - 2) \nabla_{\nu_{\imath_{\xi}^{*}g}} \nabla_{\nu_{\imath_{\xi}^{*}g}} (\imath_{\xi}^{*}\phi) -2(n - 1) \Delta_{g} \phi + 2 \Delta_{\imath_{\xi}^{*}g} \left( \imath_{\xi}^{*}\phi \right) \right) \bigg|_{X_{\zeta, 0}} \\
& \qquad \qquad + e^{-2\imath_{\zeta}^{*}\imath_{\xi}^{*}\phi} \left( - (n - 2)(n - 1) \lvert \nabla_{g} \phi \rvert_{g}^{2} + 2(n - 2) \lvert \nabla_{\imath_{\xi}^{*}g} (\imath_{\xi}^{*}\phi) \rvert_{\imath_{\xi}^{*}g}^{2} \right) \bigg|_{X_{\zeta, 0}} \\
& \qquad \qquad \qquad + e^{-2\imath_{\zeta}^{*}\imath_{\xi}^{*}\phi} \left( - 2(n - 1) \left( \nabla_{\nu_{g}} \phi \right)^{2} - 2(n - 2) \left( \nabla_{\nu_{\imath_{\xi}^{*}g}} (\imath_{\xi}^{*}\phi) \right)^{2} \right) \bigg|_{X_{\zeta, 0}}.
\end{split}
\end{equation}
Here $ -\Delta_{g}, -\Delta_{\imath_{\xi}^{*}g} $ are positive definite Laplace-Beltrami operators with respect to $ g , \imath_{\xi}^{*}g $. By (\ref{intro:eqn3}), we expect to construct a second order partial differential equation with leading differential operator closely related to the operator $ 2(n - 1)\nabla_{\nu_{g}} \nabla_{\nu_{g}} \phi + 2(n - 2) \nabla_{\nu_{\imath_{\xi}^{*}g}} \nabla_{\nu_{\imath_{\xi}^{*}g}} (\imath_{\xi}^{*}\phi) -2(n - 1) \Delta_{g} \phi + 2 \Delta_{\imath_{\xi}^{*}g} \left( \imath_{\xi}^{*}\phi \right) $ in (\ref{intro:eqn3}), with a large enough positive inhomogeneous term on $ X_{\zeta, 0} $. We need to get a smooth solution of the partial differential equation in some appropriate space, and such that the solution has small enough $ \calC_{1, \alpha} $-norm. By restricting the differential relation into $ X \cong X_{\zeta, 0} $, it then follows that the largeness of the inhomogeneous term of the partial differential equation dominates all other terms in (\ref{intro:eqn3}), hence the positivity of $ R_{\imath_{\zeta}^{*} \imath_{\xi}^{*}\tilde{g}} $ can be achieved.

Unfortunately, we cannot have all above requirements simultaneously in either $ X \times \R^{2} $ or any submanifolds of it. For example, the operator is not elliptic unless in the space $ X_{\zeta, 0} $; but if we cater to the ellipticity and the largeness of the inhomogeneous term, then the $ \calC_{1, \alpha} $-smallness of the solution is not possible; etc.

Inspired by our recent works in transposing the positivities from ambient spaces to submanifolds \cite{RX2}, \cite{XU11}, \cite{XU10}, we introduce an auxiliary one-dimensional space $ \mathbb{S}^{1} $, labeled by global $ t $-variable, to resolve the issues. Analogously, we need to impose some conformally invariant geometric condition--an angle condition along $ X_{\zeta, 0} $ with respect to the metric $ \imath_{\xi}^{*}g $ in $ X_{\xi, 0} \cong X \times \R_{\zeta} $. To some extent, our geometric conditions (\ref{intro:eqn1a}) and (\ref{intro:eqn4a}) are substitutions of the topological obstructions in e.g. \cite{CRZ}, \cite{HPS}, \cite{JR}, \cite{RosSto}, \cite{Zeilder2}. However, instead of topological restrictions, we impose a conformally invariant geometric condition, which applies to all dimensions. An immediate example of the metric that satisfies the angle condition in $ X \times \R^{2} $ is of the form $ h \oplus d\xi^{2} \oplus d\zeta^{2} $ where $ h $ is a Riemannian metric on $ X $. The following diagram summarizes the introduction of the auxiliary space:
\begin{equation}\label{intro:diagram}
\begin{tikzcd} (M : = X \times \R^{2}, g) \arrow[r, "\tau_{1}"] & ( M \times \mathbb{S}^{1} = X \times \R^{2} \times \mathbb{S}^{1}_{t}, \bar{g} : = g \oplus dt^{2}) \arrow[d, shift left = 1.5ex, "\pi"] \\
(X, \imath_{\zeta}^{*}\imath_{\xi}^{*}g) \arrow[r,"\tau_{2}"]\arrow[u,"\jmath"] &  (W : = X \times \mathbb{S}^{1}_{t}, \sigma^{*}(\bar{g})) 
\arrow[u,"\sigma"]
\end{tikzcd}
\end{equation}
Here we assign the $ t $-variable to the circle $ \mathbb{S}^{1} $ with standard metric $ dt^{2} $. We will interchangeably use $ \mathbb{S}^{1} = \mathbb{S}_{t}^{1} $ for clarity. We always identify $ W \cong W \times \lbrace 0 \rbrace_{\zeta} \times \lbrace 0 \rbrace_{\xi} \subset M \times \mathbb{S}^{1}_{t} $. Denote the Riemannian metric on $ M \times \mathbb{S}^{1} $ by $ \tilde{g} : = g \oplus dt^{2} $. We define $ \sigma : W \rightarrow M \times \mathbb{S}^{1}_{t} $ is given by $ \sigma(w) = (w, 0, 0) $ with fixed point $ (0, 0) \in \R_{\zeta} \times \R_{\xi} $. Fix a point $ P \in \mathbb{S}^{1} $, we also define $ \tau_{1} : M \rightarrow M \times \mathbb{S}^{1}, \tau_{2} : X \rightarrow W $ by sending $ M \ni x \mapsto (x, P) = \tau_{1}(x), X \ni y \mapsto (y, P) = \tau_{2}(y) $, respectively. Choose a local chart $ (U, \Psi) $ that containing $ P $, we may identify $ P $ with the point $ 0 = \Psi(P) $, labeled by $ \lbrace P \rbrace_{t} $, or $ \lbrace 0 \rbrace_{0} $ in local coordinates.

To resolve the question in Step I, we construct an elliptic partial differential equation in the closed manifold $ W $, where the angle condition is applied to ensure the ellipticity. Such a partial differential equation satisfies all requirements. We then use the solution to construct the conformal factor on $ M $, then apply (\ref{intro:eqn3}).

Once we get the positivity of $ R_{\imath_{\zeta}^{*} \imath_{\xi}^{*} \tilde{g}} $ by conformal transformation $ g \mapsto e^{2\phi} g = \tilde{g} $, Step II is straightforward: If $ g(J \cdot, \cdot) = \omega(\cdot, \cdot) $ for some Hermitian metric, then so is $ \tilde{g} $, since $ g(J\cdot, J\cdot) = g(\cdot, \cdot) $ definitely implies $ \tilde{g}(J\cdot, J\cdot) = \tilde{g}(\cdot, \cdot) $ by multiplying the conformal factor $ e^{2\phi} $. Step III then follows by results for closed Hermitian manifolds.

% Insert Stuff Here.
We point out that the angle conditions (\ref{intro:eqn1a}), (\ref{intro:eqn1b}), (\ref{intro:eqn4}) are computable. For example, with a choice of local coordinates $ (x^{1}, \dotso, x^{n - 1}, x^{n} = \zeta) $ in an open subset of $ X \times \R_{\zeta} $, we have $ \imath_{\xi}^{*}g_{ij} = \imath_{\xi}^{*}g (\partial_{i}, \partial_{j}) $, hence locally
\begin{equation*}
    \sec^{2} (\angle_{\imath_{\xi}^{*}g}(\nu_{\imath_{\xi}^{*}g}, \partial_{\zeta}))   = \frac{\imath_{\xi}^{*}g(\partial_{\zeta},\partial_{\zeta})}{\imath_{\xi}^{*}g(\nu_{\imath_{\xi}^{*}g},\partial_{\zeta})^{2}} = \frac{\sum_{i, j = 1}^{n} \left( (\imath_{\xi}^{*}g)_{ij}(\imath_{\xi}^{*}g)^{in} (\imath_{\xi}^{*}g)^{jn} \right)(\imath_{\xi}^{*}g)_{nn}}{\sum_{i = 1}^{n} \left( (\imath_{\xi}^{*}g)_{in}(\imath_{\xi}^{*}g)^{in} \right)}.
\end{equation*}
Such a quantity is mainly determined by the last row and last column of matrix of the local representation of the Riemannian metrics. Hence, for instance, the angle condition (\ref{intro:eqn1b}) is satisfied if every $ (\imath_{\xi}^{*}g)_{in} $ is not very large. Clearly the product metric is a trivial example. It is well known that there has no complete metric with uniformly Riemannian positive scalar curvature on $ \R^{2} $. It follows that, roughly speaking, the positivity of $ X \times \R^{2} $ is inherited from the positivity of $ X $ if the the Riemannian metric $ g $ on $ X \times \R^{2} $ does not intertwine too much between $ X $-side and $ \R^{2} $-side.

It is far from proving that the angle condition is also necessary. A good example to show the necessity of the angle condition must involve in the intertwine of the Riemannian metric between $ X $-side and the $ \R^{2} $-side, i.e. the metric is not of the product type. However, we can give a nontrivial ``quasi" example on $ X \times \R^{2} $, which indicates that our angle condition is reasonable, and hence necessary to some extent. The following example states in the sense of Rosenberg-Stolz conjecture.
\begin{example}\label{Mot:ex1}
Let $ X = \mathbb{T}^{n - 1}, n \geqslant 3 $. There exists a complete metric $ g = g_{X} \oplus g_{\R^{2}} $ that is of bounded curvature, such that $ R_{g} > 0 $ on $ X \times \R^{2} $, and the angle condition (\ref{intro:eqn1a}) fails along every hypersurfaces $ X \times \lbrace \zeta \rbrace \times \lbrace \xi \rbrace, \forall \zeta \in \R_{\zeta}, \forall \xi \in \R_{\xi} $.
\end{example} 
{\it{Solution}}. In $ (\xi, \zeta) $-coordinates, we set $ g_{\R^{2}} = a d\xi^{2} + 2b(\xi, \zeta) d\xi d\zeta + a d\zeta^{2} $, where $ a > 2 $ is a constant, and
\begin{equation*}
b(\xi, \zeta) = \frac{1}{(1 + e^{-\xi})(1 + e^{-\zeta})} \in (0, 1).
\end{equation*}
Clearly $ a^{2} - b^{2} > 0 $, it follows that $ g_{\R^{2}} $ is a Riemannian metric. Since $ g_{\R^{2}} \geqslant a g_{Euc} $, it is a complete metric. By Brioschi formula for the Gaussian curvature, we have
\begin{equation*}
K_{g_{\R^{2}}} = \frac{a^{2}e^{-\xi}e^{-\zeta}}{(1 + e^{-\xi})^{2}(1 + e^{-\zeta})^{2}(a^{2} - b(\xi, \zeta)^{2})^{2}} \in \left(0, \frac{a^{2}}{4(a^{2} - 1)^{2}} \right) \; {\rm on} \; \R^{2}.
\end{equation*}
It follows that the scalar curvature $ R_{g} > 0 $ on $ X \times \R^{2} $ and $ g $ is of bounded curvature. Applying (\ref{intro:eqn1a}) with $ \nu_{g} = \frac{-b}{\sqrt{a(a^{2} - b^{2})}} \partial_{\zeta} + \frac{\sqrt{a}}{\sqrt{a^{2} - b^{2}}} \partial_{\xi} $ and $ \nu_{\imath_{\xi}^{*}g} = \partial_{\zeta} $, it follows that
\begin{equation*}
\frac{n - 1}{n - 2} \frac{\imath_{\xi}^{*}g(\partial_{\zeta},\partial_{\zeta})}{\imath_{\xi}^{*}g((d\pi_{\xi})_{*}\nu_{g},\partial_{\zeta})^{2}} + \frac{\imath_{\xi}^{*}g(\partial_{\zeta},\partial_{\zeta})}{\imath_{\xi}^{*}g(\nu_{\imath_{\xi}^{*}g},\partial_{\zeta})^{2}} = \frac{a^{2} - b^{2}}{b^{2}} \cdot \frac{n - 1}{n - 2} + 1 > 3 \cdot \frac{n - 1}{n - 2} + 1 > \frac{n - 1}{n - 2} + 2.
\end{equation*}
Hence the angle condition fails on every hypersurfaces $ X \times \lbrace \zeta \rbrace \times \lbrace \xi \rbrace $. $ \Box $
\medskip

\begin{remark}\label{Mot:re1}
The example we gave in Example \ref{Mot:ex1} is not perfect, as the scalar curvature is not uniformly positive. However it indicates the necessity of the angle conditions to some extent. In $ \R^{2} $, uniformly positive scalar curvature is the same as uniformly positive Ricci curvature. By Myers' theorem, every complete manifold with positive Ricci curvature bound is compact. Therefore, if $ X $ admits no PSC metric, then the product metric $ g_{X} \oplus g_{\R^{2}} $ either fails to be complete, or fails to have uniformly positive scalar curvature. However, our example is almost there.

We can give another ``quasi" example: If $ X $ admits $ g_{X} $ with negative scalar curvature, then with the same $ g_{\R^{2}} = ad\xi^{2} + 2b(\xi, \zeta) d\xi d\zeta + a d\zeta^{2} $, we can increase $ a $ so that the scalar curvature $ R_{g} > 0 $ on $ X \times K $ for arbitrary compact subset $ K \subset \R^{2} $. Since $ \R^{2} $ admits a compact exhaustion, this gives a complete metric whose scalar curvature is positive ``almost everywhere", i.e. $ R_{g} \leqslant 0 $ only in $ \R^{2} \backslash K $, where $ K $ can be chosen as closed ball with arbitrary large radius. $ a $ should be large enough to obtain the ``almost everywhere" PSC metric, hence the angle condition must fail.

We point out that it is very hard to construct a concrete example where the positivity is coming from the intertwine between $ X $-side and $ \R $-side.
\end{remark}

% This is the Technicality Section.
\section{The Preliminaries}
In this section, we give the basic setup and all technical results that are essential for our main result in later sections. Precisely speaking, we need to construct an appropriate conformal factor on $ X \times \R^{2} $. Such a conformal factor comes from a solution of some partial differential equation. To begin with, we construct an appropriate partial differential equations in $ W = X \times \mathbb{S}_{t}^{1} $, verify that the differential operator is an elliptic operator, then show that the solution has desired properties. By (\ref{intro:eqn3}), we have to use $ \Delta_{g}\phi $ in $ M = X \times \R^{2} $ and $ \Delta_{\imath_{\xi}^{*}g} (\imath_{\xi}^{*}\phi) $ in $ X \times \R_{\zeta} $, but our partial differential equation is established in $ W $, which suggests that we also need to compare different Laplacians among different spaces. We assume the familiary of the standard knowledge of Sobolev space $ W^{k, p} $, elliptic differential operators, Lax-Milgram theorem, maximum principle, and $ \calL^{p} $-elliptic regularity, see e.g. \cite{Niren4}, \cite{Aubin}. 

First of all, we introduce two global vector fields defined on $ W = X \times \mathbb{S}_{t}^{1} $, which are associated with $ (d\pi_{\xi})_{*} \nu_{g} $ and $ \nu_{\imath_{\xi}^{*}g} $ tangential to $ X_{\xi, 0} $, respectively. Choosing local coordinates $ \lbrace x^{i}, x^{n} = \zeta, x^{n+1} = \xi \rbrace $ around some point $ x \in X_{\zeta, 0} = X \times \lbrace 0 \rbrace_{\zeta} \times \lbrace 0 \rbrace_{\xi} \subset M $ with associated local frame $ \lbrace \partial_{i},  \partial_{n} = \partial_{\zeta}, \partial_{n+1} = \partial_{\xi} \rbrace $ such that $ \lbrace \partial_{i} \rbrace_{i = 1}^{n - 1} $ are tangential to $ X_{\zeta, 0} $, we can write
\begin{align*}
    (d\pi_{\xi})_{*} \nu_{g} & = a\partial_{\zeta} + \sum_{i = 1}^{n - 1} a^{i} \partial_{i}; \\
    \nu_{\imath_{\xi}^{*}g} & = b \partial_{\zeta} + \sum_{i = 1}^{n - 1} b^{i} \partial_{i}.
\end{align*}
Note that $ b $ never vanishes along $ X_{\zeta, 0} $ since $ \nu_{\imath_{\xi}^{*}g} $ is never tangential to $ X_{\zeta, 0} $. But $ a = 0 $ when $ \nu_{g} $ is parallel to $ \partial_{\xi} $ alongside $ X_{\xi, 0} $, i.e. $ \nu_{g} \in \Gamma(T\R_{\xi}) $ hence $ (d\pi_{\xi})_{*} \nu_{g} = 0 $. If $ (d\pi_{\xi})_{*} \nu_{g} \neq 0 $, then $ a $ never nanishes also, since $ \nu_{g} $ is never tangential to $ X_{\xi, 0} $. Clearly in local coordinates,
\begin{align*}
    1 & = \imath_{\xi}^{*}g((d\pi_{\xi})_{*} \nu_{g}, (d\pi_{\xi})_{*} \nu_{g}) = g(a\partial_{\zeta}, (d\pi_{\xi})_{*} \nu_{g}) + \sum_{i} g(a^{i} \partial_{i}, (d\pi_{\xi})_{*} \nu_{g}) \\
    \Rightarrow & a = \begin{cases} g(\partial_{\zeta}, (d\pi_{\xi})_{*} \nu_{g})^{-1}, & (d\pi_{\xi})_{*} \nu_{g} \neq 0 \\ 0, & (d\pi_{\xi})_{*} \nu_{g} = 0 \end{cases}; \\
    1 & = \imath_{\xi}^{*}g(\nu_{\imath_{\xi}^{*}g}, \nu_{\imath_{\xi}^{*}g}) = g(b\partial_{\zeta}, \nu_{\imath_{\xi}^{*}g}) + \sum_{i} g(b^{i} \partial_{i}, \nu_{\imath_{\xi}^{*}g})  \Rightarrow b = g(\partial_{\zeta}, \nu_{\imath_{\xi}^{*}g})^{-1}.
\end{align*}
Since $ \partial_{\zeta}, (d\pi_{\xi})_{*} \nu_{g} $ and $ \nu_{\imath_{\xi}^{*}g} $ are globally defined along $ X_{\zeta, 0} $, we define
\begin{equation}\label{Pre:eqn0}
    V_{1} : = (d\pi_{\xi})_{*} \nu_{g} - a\partial_{\zeta}, V_{2} : = \nu_{\imath_{\xi}^{*}g} - b\partial_{\zeta}, V_{1}, V_{2} \in \Gamma(TX)
\end{equation}
such that locally $ V_{1} = \sum_{i = 1}^{n - 1} a^{i} \partial_{i} $ and $ V_{2} = \sum_{i = 1}^{n - 1} b^{i} \partial_{i} $. Again $ V_{1} = 0 $ when $ \nu_{g} \in \Gamma(T\R_{\xi}) $. We extend $ V_{1} $ and $ V_{2} $ to $ W $ with the pushforward map $ (d\pi_{2})_{*} $, and still denote them by $ V_{1} $ and $ V_{2} $ as global vector fields on $ W $.  When $ \nu_{g} \in \Gamma(TR_{\xi}) $, i.e. $ V_{1} \equiv 0 $, we apply the angle condition (\ref{intro:eqn1b}) and the same argument in $ X \times \R $ case applies, see \cite{XU11}. 

We begin with the ellipticity of the operator of our partial differential equation that will be introduced later. Assigning $ W $ the Riemannian metric $ \sigma^{*} \bar{g} = \sigma^{*}(g \oplus dt^{2}) $, we verify the ellipticity in the next lemma for the hard case, i.e. $ d(\pi_{\xi})_{*}g $ nowhere vanishes on $ X_{\xi, 0} $.
\begin{lemma}\label{Pre:lemma1} If $ (d\pi_{\xi})_{*}\nu_{g} $ never vanishes along $ X_{\zeta, 0} $, and
\begin{equation}\label{Pre:eqn1}
\frac{n - 1}{n - 2} \frac{\imath_{\xi}^{*}g(\partial_{\zeta},\partial_{\zeta})}{\imath_{\xi}^{*}g((d\pi_{\xi})_{*}\nu_{g},\partial_{\zeta})^{2}} + \frac{\imath_{\xi}^{*}g(\partial_{\zeta},\partial_{\zeta})}{\imath_{\xi}^{*}g(\nu_{\imath_{\xi}^{*}g},\partial_{\zeta})^{2}}<
2 + \frac{n - 1}{n - 2} \; {\rm on} \; X_{\zeta, 0},
\end{equation}
then the operator
\begin{equation*}
 L':=\frac{n - 1}{n - 2} \nabla_{V_{1}} \nabla_{V_{1}} + \nabla_{V_{2}}\nabla_{V_{2}} - \Delta_{\sigma^{*}\bar{g}}
\end{equation*}
is elliptic on $ W $.
\end{lemma}
\begin{proof}
In any local coordinate chart containing $ x_{0} \in X $,
 \begin{equation*}
    \nabla_{V_{1}}\nabla_{V_{1}} = \sum_{i, j = 1}^{n - 1} a^{i}(x) a^{j}(x) \frac{\partial^{2}}{\partial x^{i} \partial x^{j}} + G_{1}(x), \nabla_{V_{2}}\nabla_{V_{2}} = \sum_{i, j = 1}^{n - 1} b^{i}(x) b^{j}(x) \frac{\partial^{2}}{\partial x^{i} \partial x^{j}} + G_{2}(x)
\end{equation*}
where $ G_{1}(x), G_{2}(x) $ are linear first order operators. In Riemannian normal coordinates $(x^1,\ldots,x^{n-1})$ centered at a fixed $x_0\in X$, we have
\begin{equation*}
    \Delta_{\sigma^{*}\bar{g}}|_{x_{0}} = \sum_{i = 1}^{n - 1} \frac{\partial^{2}}{\partial \left( x^{i} \right)^{2} } + \frac{\partial^{2}}{\partial t^{2}}.
\end{equation*}
The principal symbol of $L'$ at $ x_{0} $ is
\begin{equation*}
 \sigma_2(L')(x_0, \xi) =   - \frac{n - 1}{n - 2} \sum_{i, j = 1}^{n - 1} a^{i}(x_{0})a^{j}(x_{0})- \sum_{i, j = 1}^{n - 1} b^{i}(x_{0})b^{j}(x_{0})
      \xi^{i} \xi^{j} 
     + \sum_{i = 1}^{n - 1} \left(\xi^{i} \right)^{2} + (\xi^n)^2.
\end{equation*}
$ \sigma_2(L')(x_0, \xi)>0$ for $\xi\neq 0$ implies the ellipticity of $ L' $, which is equivalent to show that
\begin{align*}
    C : & = - 
    \frac{n - 1}{n - 2} \begin{pmatrix} (a^{1})^{2} & a^{1}a^{2} &\dotso & a^{1}a^{n - 1} \\ a^{2}a^{1} &( a^{2})^{2} & \dotso & a^{2}a^{n - 1} \\ \vdots & \vdots & \ddots & \vdots \\ a^{n - 1}a^{1} & a^{n - 1}a^{2} & \dotso & (a^{n - 1})^{2} \end{pmatrix} -  \begin{pmatrix} (b^{1})^{2} & b^{1}b^{2} &\dotso & b^{1}b^{n - 1} \\ b^{2}b^{1} &( b^{2})^{2} & \dotso & b^{2}b^{n - 1} \\ \vdots & \vdots & \ddots & \vdots \\ b^{n - 1}b^{1} & b^{n - 1}b^{2} & \dotso & (b^{n - 1})^{2} \end{pmatrix} + I_{n - 1} \\
    & : = \frac{n - 1}{n - 2} A' + B' + I_{n - 1}
\end{align*}
is positive definite. Writing $ I_{n - 1} = C_{1} I_{n - 1} + C_{2} I_{n - 1} $ for any $ C_{1}, C_{2} \in (0, 1) $ with $ C_{1} + C_{2} = 1 $, it suffices to show that the two matrices
\begin{equation*}
    A: = \frac{n - 1}{n - 2} A' + C_{1} I_{n - 1}, B : = B' + C_{2} I_{n - 1}
\end{equation*}
are both positive definite. We work on $ B $ first. By Sylvester's criterion, it suffices to show that is positive definite iff all $ k \times k $ principal minors $ B_{k} $ of the symmetric matrix $ B $ have positive determinants. We claim that
\begin{equation}\label{Pre:eqn2}
 \det(B_{k}) = C_{2}^{k}\left(C_{2} + {\rm Tr}(B_{k}') \right) = C_{2}^{k}\left(C_{2} -  \sum_{i = 1}^{k} (b^{i})^{2} \right), k = 1, \dotso, n - 1.
\end{equation}
The following argument essentially comes from \cite[Lemma 2.1]{RX2}. We show this for $ B_{n-1} = B $, as the argument for the other $ B_{k} $ 
is identical. We first observe that if  $ b^{i} \neq 0, \forall i, $ then all rows of $ B' $ are proportional to each other. It follows that the kernel of $B'$ has dimension $ n - 2 $,  so $0$ is an eigenvalue of  $ B' $ of multiplicity $ n - 2 $. Since the sum of  the eigenvalues of $ B' $ is  the trace $- \sum_{i = 1}^{n - 1} (b^{i})^{2} $, 
the nontrivial eigenvalue must be $- \sum_{i = 1}^{n - 1} (b^{i})^{2} $.
Thus the eigenvalues of $B$ are 
\begin{equation*}
    \lambda_{1} = - \sum_{i = 1}^{n - 1} (b^{i})^{2} + C_{2}, \lambda_2 = \ldots = \lambda_{n-1} = C_{2},
\end{equation*}
which proves (\ref{Pre:eqn2}) by Jordan decomposition of symmetric matrix. If some  $ b^{i} $ vanish, then the claim reduces to a lower dimensional case by removing the corresponding  rows and columns of all zeros. 

Since the matrix $ A' $ is of the same form of $ B' $, the same argument implies that
\begin{equation}\label{Pre:eqn3}
 \det(A_{k}) = C_{1}^{k}\left(C_{1} + \frac{n - 1}{n - 2}{\rm Tr}(A_{k}') \right) = C_{1}^{k}\left(C_{1} -  \frac{n - 1}{n - 2} \sum_{i = 1}^{k} (a^{i})^{2} \right), k = 1, \dotso, n - 1
\end{equation}
where $ A_{k}'s $ are the $ k \times k $ principal minors of the symmetric matrix $ A $. Writing $ C_{2} = 1 - C_{1} $, $ A $ and $ B $ are positive definite, or equivalently the quanitites in (\ref{Pre:eqn2}) and (\ref{Pre:eqn3}) are all positive, if
\begin{equation}\label{Pre:eqn4}
    \frac{n - 1}{n - 2} \sum_{i = 1}^{n - 1}\left( a^{i} \right)^{2} < C_{1} \; {\rm and} \; \sum_{i = 1}^{n - 1} \left( b^{i} \right)^{2} < 1 - C_{1} \Leftrightarrow \frac{n - 1}{n - 2} \sum_{i = 1}^{n - 1}\left( a^{i} \right)^{2}  + \sum_{i = 1}^{n - 1} \left( b^{i} \right)^{2} < 1.
\end{equation}
In normal coordinates,
\begin{align*}
& \frac{n - 1}{n - 2} \sum_{i = 1}^{n - 1}\left( a^{i} \right)^{2}  + \sum_{i = 1}^{n - 1} \left( b^{i} \right)^{2} \\
& \qquad = \frac{n - 1}{n - 2} \lvert V_{1} \rvert_{\sigma^{*} \bar{g}}^{2} + \lvert V_{2} \rvert_{\sigma^{*} \bar{g}}^{2} \\
& \qquad = \frac{n - 1}{n - 2} \imath_{\xi}^{*}g \left((d\pi_{\xi})_{*}\nu_{g} - a \partial_{\zeta}, (d\pi_{\xi})_{*}\nu_{g} - a \partial_{\zeta} \right) + \imath_{\xi}^{*}g \left(\nu_{\imath_{\xi}^{*}g} - b \partial_{\zeta}, \nu_{\imath_{\xi}^{*}g} - b \partial_{\zeta} \right) \\
& \qquad = \frac{n - 1}{n - 2} \left( 1 - 2 + \frac{\imath_{\xi}^{*}g(\partial_{\zeta},\partial_{\zeta})}{\imath_{\xi}^{*}g((d\pi_{\xi})_{*}\nu_{g},\partial_{\zeta})^{2}} \right) + \left( 1 - 2 + \frac{\imath_{\xi}^{*}g(\partial_{\zeta},\partial_{\zeta})}{\imath_{\xi}^{*}g(\nu_{\imath_{\xi}^{*}g},\partial_{\zeta})^{2}} \right).
\end{align*}
Therefore (\ref{Pre:eqn4}) holds if and only if (\ref{Pre:eqn1}) holds, provided that $ (d\pi_{\xi})_{*}\nu_{g} $ nowhere vanish on $ X_{\zeta, 0} $.
\end{proof}
\begin{remark}\label{Pre:re1}
When $ (d\pi_{\xi})_{*}\nu_{g} \equiv 0 $ on $ X_{\xi, 0} $, the condition (\ref{Pre:eqn1}) degenerates to
\begin{equation*} \frac{\imath_{\xi}^{*}g(\partial_{\zeta},\partial_{\zeta})}{\imath_{\xi}^{*}g(\nu_{\imath_{\xi}^{*}g},\partial_{\zeta})^{2}} < 2 
\end{equation*}
for the operator $ L' = \nabla_{V_{2}}\nabla_{V_{2}} - \Delta_{\sigma^{*}\bar{g}} $ in $ W $. Equivalently, it says $ \angle_{\imath_{\xi}^{*}g}(\nu_{\imath_{\xi}^{*}g},\partial_{\zeta}) \in [0, \frac{\pi}{4}) $, as we set in \cite{RX2}, \cite{XU11}, \cite{XU10}. Consequently, our problem reduces exactly to the case of \cite[Theorem 3.1]{XU11}. From now on, we assume that $ (d\pi_{\xi})_{*}\nu_{g} $ never vanishes on $ X_{\xi, 0} $. 
\end{remark}
Next we construct an appropriate inhomogenous term of our partial differential equation, as the next lemma shows.
\begin{lemma}\label{Pre:lemma2} 
Fix $ p \in \mathbb{N} $,  $ C \gg 1 $ and any positive $ \delta \ll 1 $, there exists a positive smooth function $F:W\to\R$ and small enough constant $ 0 < \epsilon \ll 1 $ such that such that (i) $ [-\epsilon, \epsilon]_{t} \subset \Psi(U) $; (ii) $ F |_{X_{\zeta, 0} \times (-\frac{\epsilon}{2}, \frac{\epsilon}{2})_{t}} = C+1 $, (iii) $ F \equiv 0 $ outside $ X_{\zeta, 0} \times [-\epsilon, \epsilon]_{t} $; and (iv) $ \lVert  F  \rVert_{\calL^{p}(W, \sigma^{*}\bar{g})} < \delta $.
\end{lemma}
\begin{proof}
In a word, we first set a function that is large enough in the tubular neighborhood $ X_{\zeta, 0} \times \left(-\frac{\epsilon}{2}, \frac{\epsilon}{2} \right)_{t} $, and decay to zero fast enough by letting $ \epsilon $ to be small enough. Exactly the same argument follows from \cite[Lemma 2.2]{XU11}. 
\end{proof}
We define the $\calC^{1, \alpha} $-norm on $ W $ for any $ \alpha \in (0, 1) $  by fixing a chart cover $ W = \bigcup_{i} (U_{i}, \varphi_{i}) $ such that
\begin{equation*}
    \lVert u \rVert_{\calC^{1, \alpha}(W)} = \lVert u \rVert_{\calC^{0}(W)} + \sup_{i, k} \lVert \partial_{x_{i}^{k}}u \rVert_{\calC^{0}(U_{i})} + \sup_{i, k} \sup_{x \neq x', x, x' \in U_{i}} \frac{\left\lvert \partial_{x_{i}^{k}}u(x) - \partial_{x_{i}^{k}}u(x') \right\rvert}{\lvert x - x' \rvert^{\alpha}}
\end{equation*}
where $ \partial_{x_{i}^{k}} u $ are local representations in $ U_{i} $ with local coordinates $ \lbrace x_{i}^{1}, \dotso, x_{i}^{n - 1}, x_{i}^{n + 2} = t \rbrace $ and associated local frames $ \lbrace \partial_{x_{i}^{k}} \rbrace $. 

Let $ R_{\bar{g}} $ be the Riemannian scalar curvature of $(M \times \mathbb{S}^{1}, \bar{g}) $. We now introduce the partial differential equation in $ W $ and verify $ \calC_{1, \alpha} $-smallness of the solution. Such a solution will be used to construct our conformal factor. The next result is essentially the same as \cite[Proposition 2.1]{XU11}. 
\begin{lemma}\label{Pre:lemma3}
Let $ (M, g) = (X \times \R^{2}, g) $ be a noncompact complete manifold such $ R_{g} \geqslant \kappa_{0} > 0 $ for some $ \kappa_{0} \in \R $. Let $ (W, \sigma^{*} \bar{g}) $ be the closed manifold as above. Assume that (\ref{Pre:eqn1}) holds. For any positive constant $ \eta \ll 1 $, any positive constant $ C $, and any $ p > n = \dim W $, there exists an associated $ F $ and $ \delta $ in the sense of Lemma \ref{Pre:lemma2}, such that the following partial differential equation     
\begin{equation}\label{Pre:eqn5}
L u : = \frac{4(n - 1)}{n - 2}\nabla_{V_{1}}\nabla_{V_{1}}u + 4\nabla_{V_{2}} \nabla_{V_{2}} u - 4\Delta_{\sigma^{*} \bar{g}}  u + R_{\bar{g}} |_{W}u =
F  \; {\rm in} \; W
\end{equation}
admits a unique smooth solution $ u $ with
\begin{equation}\label{Pre:eqn6}
\lVert u \rVert_{\calC^{1, \alpha}(W)} < \eta
\end{equation}
for some $ \alpha \in (0, 1) $ such that $ \alpha \geqslant 1 - \frac{n}{p} $.
\end{lemma}
\begin{proof}
Since $ X $ is compact, (\ref{Pre:eqn1}) implies that the maximum of the continuous function
\begin{equation*}
    \frac{n - 1}{n - 2} \frac{\imath_{\xi}^{*}g(\partial_{\zeta},\partial_{\zeta})}{\imath_{\xi}^{*}g((d\pi_{\xi})_{*}\nu_{g},\partial_{\zeta})^{2}} + \frac{\imath_{\xi}^{*}g(\partial_{\zeta},\partial_{\zeta})}{\imath_{\xi}^{*}g(\nu_{\imath_{\xi}^{*}g},\partial_{\zeta})^{2}}
\end{equation*}
is also strictly less than $ 2 + \frac{n - 1}{n - 2} $, and hence the uniform ellipticity of the operator $ L $ is obtained by Lemma \ref{Pre:lemma1}. 

Since $ X $ is a closed manifold, so is $ W $. Fix $ \eta \ll 1, C $, and $ p > n $. We then fix some $ \alpha \in (0, 1) $ such that $ 1 + \alpha \geqslant 2 - \frac{n}{p} $. Denote $ C = C(W, g, n, p, L) $ by the constant of $ \calL^{p} $ elliptic regularity estimates with respect to $ L $, and $ C' = C'(W, g, n, p, \alpha) $ by the constant of the Sobolev embedding $ W^{2, p} \hookrightarrow \calC^{1, \alpha} $ over closed manifolds. Fix $ \delta $ such that $ \delta C C' < \eta $. Finally, we choose an associated $ F $ in the sense of Lemma \ref{Pre:lemma2}.

By assumption, $ R_{g} \geqslant \kappa_{0} > 0 $ on $ M $ for some positive constant $ \kappa_{0} $, therefore $ R_{\bar{g}} > 0 $ on $ M \times \mathbb{S}^{1}_{t} $, hence $ R_{\bar{g}} |_{W} > 0 $ uniformly. By the maximum principle for closed manifolds, the elliptic operator $ L $ is an injective operator. It follows that there exists a unique solution $ u \in H^{1}(W, \sigma^{*}\bar{g}) $ of (\ref{Pre:eqn5}) by Fredholm dichotomy. By standard $ H^{s} $-type elliptic theory and the smoothness of $ F $, $ u \in \calC^{\infty}(W) $. 

By standard $ \calL^{p} $-regularity theory \cite{Niren4}, 
\begin{equation*}
    \lVert u \rVert_{W^{2, p}(W, \sigma^{*} \bar{g})} \leqslant C \left( \lVert F \rVert_{\calL^{p}(W, \sigma^{*} \bar{g})} + \lVert u \rVert_{\calL^{p}(W, \sigma^{*}\bar{g})} \right).
\end{equation*}
Due to the injectivity of the operator, a very similar argument of \cite[Proposition 2.1]{RX2} shows that $ \lVert u \rVert_{\calL^{p}(W, \sigma^{*} \bar{g})} $ can be bounded above by $ \lVert Lu \rVert_{\calL^{p}(W, \sigma^{*}\bar{g})} $. It follows that the $ \calL^{p}$ estimates of our solution of (\ref{Pre:eqn5}) can be improved by
\begin{equation*}
     \lVert u \rVert_{W^{2, p}(W, \sigma^{*} \bar{g})} \leqslant C \lVert F \rVert_{\calL^{p}(W, \sigma^{*} \bar{g})}.
\end{equation*}
By Sobolev embedding, it follows that
\begin{equation*}
    \lVert u \rVert_{\calC^{1, \alpha}(W)} \leqslant C' \lVert u \rVert_{W^{2, p}(W, \sigma^{*}\bar{g})} \leqslant CC'\lVert F \rVert_{\calL^{p}(W, \sigma^{*} \bar{g})} < CC' \delta < \eta. 
\end{equation*}
\end{proof}
\begin{remark}\label{Pre:re0}
Note that $ V_{1} \equiv 0 $ if  $ d(\pi_{\xi})_{*}\nu_{g} \equiv 0 $ on $ X_{\xi, 0} $. In this case, our partial differential equation (\ref{Pre:eqn5}) is reduced to
\begin{equation*}
    4\nabla_{V_{2}} \nabla_{V_{2}} u - 4\Delta_{\sigma^{*}\bar{g}} + R_{\bar{g}} |_{W} u = F \; {\rm in} \; W,
\end{equation*}
which is associated to the reduced angle condition in Remark \ref{Pre:re1}. The same $ \calC^{1, \alpha} $-smallness in Lemma \ref{Pre:lemma3} for the solution of this revised partial differential equation are given in \cite[Proposition 2.1]{XU11}.
\end{remark}
We now construct our conformal factor on $ M $. Denote by $ X \times \R_{\zeta} = Y $. First of all, we can choose $ \eta \ll 1 $ such that
\begin{equation*}
    u_{W} : = u + 1 \in \left[ \frac{1}{2}, \frac{3}{2} \right] \; {\rm on} \; W.
\end{equation*}
Recall the diagram (\ref{intro:diagram}) here and an analogous one with respect to $ Y = X \times \R_{\zeta} $:
\begin{equation}\label{Pre:diagram}
\begin{tikzcd} (M, g) \arrow[r, "\tau_{1}"] & ( M \times \mathbb{S}^{1}, \bar{g}) \arrow[d, shift left = 1.5ex, "\pi"] \\
(X, \imath_{\zeta}^{*}\imath_{\xi}^{*}g) \arrow[r,"\tau_{2}"]\arrow[u,"\jmath"] &  (W, \sigma^{*}(\bar{g})) 
\arrow[u,"\sigma"]
\end{tikzcd}, 
\begin{tikzcd} (Y, g) \arrow[r, "\tau_{1}'"] & ( Y \times \mathbb{S}^{1}, \imath_{\xi}^{*}g \oplus dt^{2}) \arrow[d, shift left = 1.5ex, "\pi'"] \\
(X, \imath_{\zeta}^{*}\imath_{\xi}^{*}g) \arrow[r,"\tau_{2}"]\arrow[u,"\imath_{\zeta}"] &  (W, \sigma^{*}(\bar{g})) 
\arrow[u,"\sigma'"]
\end{tikzcd}.
\end{equation}
Note that $ \jmath, \pi, \sigma, \pi', \sigma' $ defined in (\ref{Pre:diagram}) satisfy $ \jmath = \imath_{\xi} \circ \imath_{\zeta} $, $ \pi = (\pi_{\zeta} \circ \pi_{\xi}) \times Id_{\mathbb{S}_{t}^{1}} $, $ \sigma = \jmath \times Id_{\mathbb{S}_{t}^{1}} $, $ \pi' = \pi_{\xi} \times Id_{\mathbb{S}_{t}^{1}} $, and $ \sigma' = \imath_{\zeta} \times Id_{\mathbb{S}_{t}^{1}} $. We define
\begin{equation}\label{Pre:eqn7}
\begin{split}
\hat{u} & : = \pi^{*}u_{W} : M \times \mathbb{S}^{1} \rightarrow \R, u_{M} : = \tau_{1}^{*} \hat{u} : M \rightarrow \R, u_{X} : = \tau_{2}^{*} u_{W} : X \rightarrow \R; \\
\hat{u}' & : = (\pi')^{*} u_{W} : Y \times \mathbb{S}^{1} \rightarrow \R, u_{Y} : = (\tau_{1}')^{*} \hat{u}' : Y \rightarrow \R.
\end{split}
\end{equation}
By the commutative diagram (\ref{Pre:diagram}), we have
\begin{equation}\label{Pre:eqn8}
\sigma^{*} \hat{u} = u_{W}, (\sigma')^{*} \hat{u}' = u_{W}, u_{Y} = \imath_{\xi}^{*} u_{M}, u_{X} = \tau_{2}^{*} u_{W} = \tau_{2}^{*} \sigma^{*} \hat{u} = \jmath^{*} u_{M} = \jmath^{*} \tau_{1}^{*} \hat{u} = \imath_{\zeta}^{*} u_{Y}.
\end{equation}
Recall that $ n - 1 = \dim_{\R} X \geqslant 2 $. The function $ u_{M}^{\frac{4}{n - 2}} $ will be used as the conformal factor of the original metric $ g $, where $ u_{Y}^{\frac{4}{n -2}} $ will be used as the conformal factor of the metric $ \imath_{\xi}^{*}g $ correspondingly.

For later use, we need to pullback $ u $ to $ M \times \mathbb{S}^{1} $ also. Define
\begin{equation*}
    u_{1} : = \pi^{*}u : M \times \mathbb{S}^{1} \rightarrow \R.
\end{equation*}
Fixing a sufficiently large closed interval $ I_{\xi} = [-k, k]_{\xi} $ such that $ [-1, 1] \subset I_{\xi} \in \R_{\xi} $, we define a smooth, nonnegative function
\begin{equation*}
    \varphi : \R_{\xi} \rightarrow \R, \varphi \geqslant 0 \; {\rm on} \; R_{\xi}, \varphi(\xi) = 1 \; {\rm on} \; [-1, 1], \varphi(\xi) \equiv 0 \; {\rm on} \; \R_{\xi} \backslash I_{\xi}.
\end{equation*}
Similarly, we define another smooth, nonnegative function with $ I_{\zeta} = [-k, k]_{\zeta} $
\begin{equation*}
    \varphi' : \R_{\zeta} \rightarrow \R, \varphi' \geqslant 0 \; {\rm on} \; R_{\zeta}, \varphi'(\zeta) = 1 \; {\rm on} \; [-1, 1], \varphi'(\zeta) \equiv 0 \; {\rm on} \; \R_{\zeta} \backslash I_{\zeta}.
\end{equation*}
With natural projections $ \Pi_{1} : M \times \mathbb{S}_{t}^{1} \rightarrow \R_{\xi} $ and $ \Pi_{2} : M \times \mathbb{S}_{t}^{1} \rightarrow \R_{\zeta} $, we obtain the pullback functions
\begin{equation}\label{Pre:eqn9}
    \Phi:= \Pi_{1}^{*}\varphi : M \times \mathbb{S}_{t}^{1} \rightarrow \R, \Phi' : = \Pi_{2}^{*} \varphi' : M \times \mathbb{S}_{t}^{1} \rightarrow \R.
\end{equation}
such that $ \Phi \geqslant 0 $ on $ M \times \mathbb{S}_{t}^{1} $, $ \text{supp}(\Phi) \subset X \times \R_{\zeta} \times \mathbb{S}_{t}^{1} \times I_{\xi} $, $ \Phi' \geqslant 0 $ on $ M \times \mathbb{S}_{t}^{1} $, $ \text{supp}(\Phi') \subset X \times \R_{\xi} \times \mathbb{S}_{t}^{1} \times I_{\zeta} $, $ \Phi \cdot \Phi' $ is compactly supported, and $ \Phi \cdot \Phi' = 1 $ on each hypersurface $ W \times \lbrace \xi \rbrace \times \lbrace \zeta \rbrace = X \times \mathbb{S}_{t}^{1} \times \lbrace \xi \rbrace \times \lbrace \zeta \rbrace $ when $ \xi \in [-1, 1]_{\xi}, \zeta \in [-1, 1]_{\zeta} $. 
Define
\begin{equation*}
    \bar{u} : = \Phi \cdot \Phi' \cdot u_{1} : M \times \mathbb{S}^{1} \rightarrow \R.
\end{equation*}
By (\ref{Pre:eqn9}), we observe that
\begin{equation*}
    \bar{u} |_{W} = u.
\end{equation*}
\begin{remark}\label{Pre:re2}
With these setup, we observe that $ \hat{u} $ is constant with respect to both $ \xi $- and $ \zeta $-variables, $ \hat{u}' $ is constant with respect to $ \zeta $-variable. It follows that all orders of both $ \xi $- and $ \zeta $-derivatives of $ \hat{u} $ vanish, and all orders of $ \zeta $-derivatives of $ \hat{u}' $ vanish.
\end{remark}

With the introduction of the auxiliary space $ \mathbb{S}_{t}^{1} $, extra terms like $ \frac{\partial^{2} u}{\partial t^{2}} $ arise. But the conformal transformation of Riemannian scalar curvature in $ M $ should not contain this. It follows that we need to take a partial $ \calC^{2} $-estimate to know the largeness of $ \frac{\partial^{2} u}{\partial t^{2}} $. Prior to the partial $ \calC^{2} $-estimate, we need to introduce a global notion of positivity, which will be used directly to control the $ \calC^{0} $-norm of $ \frac{\partial^{2} u}{\partial t^{2}} $ in a tubular neighborhood of $ X \times \lbrace 0 \rbrace_{t} \subset W $.

Let $ (N, \hat{g}) $ be any noncompact manifold with some Riemannian metric $ \hat{g} $ and Riemannian scalar curvature $ R_{\hat{g}} $, $ \dim_{\R} N = n $. We define
\begin{equation*}
E_{\hat{g}}(v) = \frac{\frac{4(n -1)}{n - 2} \int_{N} \lvert \nabla_{\hat{g}} v \rvert^{2} \dvoll + \int_{N} R_{\hat{g}} v^{2} \dvoll}{\left( \int_{N} v^{\frac{2n}{n - 2}} \dvoll \right)^{\frac{n-2}{n}}}, \forall v \in \calC_{c}^{\infty}(N) \backslash \lbrace 0 \rbrace.
\end{equation*}
Analogous to the compact case, the {\it{Yamabe constant}} of the conformal class $ [\hat{g}] $ on $ N $ is defined by
\begin{equation*}
    Y(N, \hat{g}) = \inf_{v \in \calC_{c}^{\infty}(N) \backslash \lbrace 0 \rbrace} E_{\hat{g}}(v).
\end{equation*}
It is known \cite[Lemma 2.1]{Wei} that for any complete noncompact manifold $ (N, \hat{g}) $, we have
\begin{equation}\label{Pre:eqn10}
    -\lVert (R_{\hat{g}})_{-} \rVert_{\calL^{\frac{n}{2}}(N, \hat{g})} \leqslant Y(N, \hat{g}) \leqslant \frac{4(n - 1)}{n -2}\Lambda
\end{equation}
where $ (R_{\hat{g}})_{-} $ is the negative part of the scalar curvature on $ N $, and $ \Lambda $ is the best Sobolev constant on $ \R^{n} $. In addition, we denote the infimum of the $ \calL^{2} $-spectrum of the conformal Laplacian on $ (N, \hat{g}) $ by
\begin{equation*}
    \mu(N, \hat{g}) = \inf_{v \in \calC_{c}^{\infty}(N) \backslash \lbrace 0 \rbrace} \frac{\frac{4(n -1)}{n - 2} \int_{N} \lvert \nabla_{\hat{g}} v \rvert^{2} \dvoll + \int_{N} R_{\hat{g}} v^{2} \dvoll}{\int_{N} v^{2} \dvoll}.
\end{equation*}
The next result is essentially due to Grosse \cite[Lemma 7]{Grosse}.
\begin{lemma}\label{Pre:lemma4}
If $ (N, \hat{g}) $ is a noncompact Riemannian manifold of bounded curvature with $ R_{\hat{g}} \geqslant \kappa_{0} > 0 $ for some $ \kappa_{0} \in \R $, then $ \mu(N, \hat{g}) > 0 $. Consequently, $ Y(N, \hat{g}) > 0 $.
\end{lemma}
\begin{proof}
Exactly the same argument follows from \cite[Lemma 2.3]{XU11}. The bounded curvature assumption implies the continuity of the Sobolev embedding \cite[Theorem 2.21]{Aubin}.
\end{proof}
We now apply the positivity of the Yamabe constant (\ref{Pre:eqn10}) to estimate $ \left\lVert \frac{\partial^{2} u}{\partial t^{2}} \right\rVert_{\calC^{0}\left(X \times \left( -\frac{\epsilon}{4}, \frac{\epsilon}{4} \right)_{t} \right) } $ with the same $ \epsilon $ given in Lemma \ref{Pre:lemma2}. The key is that the positivity of the Yamabe constant is invariant under conformal transformation, and therefore invariant under scaling. At begin with, we know nothing about the positivity of the Yamabe constant in $ W $. But the positive lower bound of the Yamabe quotient with respect to $ \bar{u} $ can be ``transferred" to the Yamabe quotient with respect to $ u $ in $ W $. By (\ref{Pre:eqn7}) and (\ref{Pre:eqn8}), the same estimate applies to $ u_{W} |_{X \times \lbrace 0 \rbrace_{\xi} \times \lbrace 0 \rbrace_{\zeta}} $ and hence $ u_{M} |_{X \times \lbrace 0 \rbrace_{\xi} \times \lbrace 0 \rbrace_{\zeta}} $.
\begin{lemma}\label{Pre:lemma5}
Choosing $ \epsilon $ defined in Lemma \ref{Pre:lemma2} to be small enough that will be determined below. Assume that $ (M, g) $ is of bounded curvature such that $ R_{g} \geqslant \kappa_{0} > 0 $ for some $ \kappa_{0} \in \R $. Let $ (W, \sigma^{*}\bar{g}) $ be the associated space defined in (\ref{Pre:diagram}). Let $ p, C, \delta $ and $ F $ be the same as in Lemma \ref{Pre:lemma3}. If $ u $ is the associated solution of (\ref{Pre:eqn5}), then there exists a constant $ \eta' \ll 1 $ such that
\begin{equation}\label{Pre:eqn11}
    \left\lVert \frac{\partial^{2} u}{\partial t^{2}} \right\rVert_{\calC^{0}\left(X \times \lbrace 0 \rbrace_{\xi} \times \lbrace 0 \rbrace_{\zeta} \times \left(-\frac{\epsilon}{4}, \frac{\epsilon}{4} \right)_{t} \right)} < \eta'.
\end{equation}
\end{lemma}
\begin{proof}
Although we need to deal with the analysis in $ X \times \R^{2} \times \mathbb{S}^{1} $ and $ X \times \mathbb{S}^{1} $, our argument is essentially due to \cite[Lemma 2.4]{XU11}. Set $ U_{1} = X \times \lbrace 0 \rbrace_{\xi} \times \lbrace 0 \rbrace_{\zeta} \times \left(-\frac{\epsilon}{4}, \frac{\epsilon}{4} \right)_{t} $, $ U_{2} =  X \times \lbrace 0 \rbrace_{\xi} \times \lbrace 0 \rbrace_{\zeta} \times \left(-\frac{\epsilon}{2}, \frac{\epsilon}{2} \right)_{t} $, $ O_{1} =  X \times \lbrace 0 \rbrace_{\xi} \times \lbrace 0 \rbrace_{\zeta} \times \left(-\frac{1}{4}, \frac{1}{4} \right) $, $ O_{2} =  X \times \lbrace 0 \rbrace_{\xi} \times \lbrace 0 \rbrace_{\zeta} \times \left(-\frac{1}{2}, \frac{1}{2} \right) $, $ O_{3} =  X \times \lbrace 0 \rbrace_{\xi} \times \lbrace 0 \rbrace_{\zeta} \times \left(-1, 1 \right) $. Since $ \bar{g} = g \oplus dt^{2} $, the scalar curvature $ R_{\bar{g}} $ is constant with $ t $-variable; in addition, the vector fields $ V_{1}, V_{2} $ and $ \Delta_{\sigma^{*} \bar{g}} $ are constant with $ t $-variable. By Lemma \ref{Pre:lemma2}, $ F \equiv C + 1 $ on $ U_{2} $. Applying $ \frac{\partial^{2}}{\partial t^{2}} $ on both sides of (\ref{Pre:eqn5}) in $ U_{2} $, it follows that 
\begin{equation*}
  L_{\sigma^{*} \bar{g}} \left( \frac{\partial^{2}u}{\partial t^{2}} \right) =  \frac{4(n - 1)}{n - 2} \nabla_{V_{1}} \nabla_{V_{1}} \left( \frac{\partial^{2}u}{\partial t^{2}} \right) - 4\nabla_{V_{2}} \nabla_{V_{2}} \left( \frac{\partial^{2}u}{\partial t^{2}} \right)- 4\Delta_{\sigma^{*} \bar{g}}  \left( \frac{\partial^{2}u}{\partial t^{2}} \right) + R_{\bar{g}} |_{W}\left( \frac{\partial^{2}u}{\partial t^{2}} \right) = 0
\end{equation*}
in $ U_{2} $. Set $ t = \epsilon t' $, the new $ t' $-variable is associated with the metric $ (dt')^{2} = \epsilon^{-2} dt^{2} $ on $ \mathbb{S}^{1} $. We have
\begin{align*}
   & (\epsilon^{-1})^{2} \left( \frac{4(n - 1)}{n - 2} \nabla_{V_{1, \epsilon^{-2}\bar{g}}} \nabla_{V_{1, \epsilon^{-2}\bar{g}}} \left( \frac{\partial^{2}u}{\partial t^{2}} \right) + 4\nabla_{V_{2, \epsilon^{-2}\bar{g}}} \nabla_{V_{2, \epsilon^{-2}\bar{g}}} \left( \frac{\partial^{2}u}{\partial t^{2}} \right) \right) \\
   & \qquad + (\epsilon^{-1})^{2} \left( - 4\Delta_{\sigma^{*} \bar{g}}  \left( \frac{\partial^{2}u}{\partial (t')^{2}} \right) + R_{\bar{g}} |_{W}\left( \frac{\partial^{2}u}{\partial (t')^{2}} \right) \right) = 0
\end{align*}
in $ O_{2} $. Equivalently, the following PDE holds in $ O_{2} $ with the new metric $ \bar{g} \mapsto \epsilon^{-2} \bar{g} $,
\begin{equation}\label{Pre:eqn12}
\begin{split}
& L_{\sigma^{*}(\epsilon^{-2}\bar{g})}\left( \frac{\partial^{2}u}{\partial (t')^{2}} \right) : =  \\
& \qquad \frac{4(n - 1)}{n - 2} \nabla_{V_{1, \epsilon^{-2}\bar{g}}} \nabla_{V_{1, \epsilon^{-2}\bar{g}}} \left( \frac{\partial^{2}u}{\partial (t')^{2}} \right) + 4\nabla_{V_{2, \epsilon^{-2}\bar{g}}} \nabla_{V_{2, \epsilon^{-2}\bar{g}}} \left( \frac{\partial^{2}u}{\partial (t')^{2}} \right) \\
& \qquad \qquad - 4\Delta_{\sigma^{*} (\epsilon^{-2}\bar{g})}  \left( \frac{\partial^{2}u}{\partial (t')^{2}} \right) + R_{\epsilon^{-2}\bar{g}} |_{W}\left( \frac{\partial^{2}u}{\partial (t')^{2}} \right) = 0.
\end{split}
\end{equation}
Here $ V_{1, \epsilon^{-2}\bar{g}} $ and $  V_{2, \epsilon^{-2}\bar{g}} $ are the vector fields with respect to the new metric $ \epsilon^{-2} \bar{g} $ satisfying $ V_{\epsilon^{-2}\bar{g}} = \epsilon^{-1} V $. With the new metric $ \epsilon^{-2} \bar{g} $, we estimate $ \lVert \frac{\partial^{2}u}{\partial (t')^{2}} \rVert_{\calC^{0}(O_{1})} $ with local $ H^{s} $-type elliptic regularity. We set $ s > 0 $ such that $ s - \frac{n}{2} = 1 + \alpha \geqslant 2 - \frac{n}{p} $, where $ p $ and $ \alpha $ are given in Lemma \ref{Pre:lemma3}. With usual metric $ \sigma^{*} \bar{g} $, the standard $ H^{s} $-type local elliptic regularity estimates for second order elliptic operator $ L_{\sigma^{*}\bar{g}} $ gives
\begin{equation*}
\lVert v \rVert_{H^{s}(U, \sigma^{*} \bar{g})} \leqslant D_{s} \left( \lVert L_{\sigma^{*} \bar{g}} u \rVert_{H^{s - 2}(V, \sigma^{*} \bar{g})} + \lVert v \rVert_{\calL^{2}(V, \sigma^{*} \bar{g})} \right), \forall v \in \calC^{\infty}(V), U \subset V.
\end{equation*}
Here the constant $ D_{s} $ only depends on the operator $ L_{\sigma^{*}\bar{g}} $, the metric, the degree $ s $ and the domain $ U, V $, and is independent of $ \epsilon $. For the scaled metric $ \sigma^{*} (\epsilon^{-2} \bar{g}) $, the second order elliptic operator $ L_{\sigma^{*}(\epsilon^{-2}\bar{g})} $ satisfies $ L_{\sigma^{*}(\epsilon^{-2}\bar{g})} = \epsilon^{2} L_{\sigma^{*}\bar{g}}$. We apply this local elliptic estimate in $ O_{1} $ for the second order elliptic operator $ L_{\sigma^{*}(\epsilon^{-2}\bar{g})} $, 
\begin{equation}\label{Pre:eqn13}
\begin{split}
\left\lVert \frac{\partial^{2}u}{\partial (t')^{2}} \right\rVert_{H^{s}(O_{1}, \sigma^{*}(\epsilon^{-2} \bar{g}))} & \leqslant \epsilon^{-\frac{n}{2}} \left\lVert \frac{\partial^{2}u}{\partial (t')^{2}} \right\rVert_{H^{s}(O_{1}, \sigma^{*} \bar{g})} \\
& \leqslant \epsilon^{-\frac{n}{2}} \left( D_{s} \left\lVert L_{\sigma^{*}\bar{g}} \left( \frac{\partial^{2}u}{\partial (t')^{2}} \right) \right\rVert_{H^{s - 2}(O_{2}, \sigma^{*} \bar{g})} + \left\lVert \frac{\partial^{2}u}{\partial (t')^{2}} \right\rVert_{\calL^{2}(O_{2}, \sigma^{*} \bar{g})} \right) \\
& \leqslant D_{s, \epsilon}  \left\lVert L_{\sigma^{*}(\epsilon^{-2}\bar{g})} \left( \frac{\partial^{2}u}{\partial (t')^{2}} \right) \right\rVert_{H^{s - 2}(O_{2}, \sigma^{*} (\epsilon^{-2}\bar{g}))} + \epsilon^{-\frac{n}{2}} D_{s} \left\lVert \frac{\partial^{2}u}{\partial (t')^{2}} \right\rVert_{\calL^{2}(O_{2}, \sigma^{*} \bar{g})} \\
& =  D_{s} \left\lVert \frac{\partial^{2}u}{\partial (t')^{2}} \right\rVert_{\calL^{2}(O_{2}, \sigma^{*}(\epsilon^{-2} \bar{g}))}.
\end{split}
\end{equation}
Here the $ H^{s - 2} $-term vanishes due to the differential relation (\ref{Pre:eqn12}). With local $ H^{2} $-type local elliptic estimate for (\ref{Pre:eqn5}) and H\"older's inequality,
\begin{equation}\label{Pre:eqn14}
    \begin{split}
        \left\lVert \frac{\partial^{2}u}{\partial (t')^{2}} \right\rVert_{\calL^{2}(O_{2}, \sigma^{*} (\epsilon^{-2} \bar{g}))} & = \epsilon^{2 - \frac{n}{2}} \left\lVert \frac{\partial^{2}u}{\partial t^{2}} \right\rVert_{\calL^{2}(O_{2}, \sigma^{*} \bar{g})} \leqslant  \epsilon^{2 - \frac{n}{2}} \lVert u \rVert_{H^{2}(O_{2}, \sigma^{*}\bar{g})} \\
        & \leqslant \epsilon^{2 - \frac{n}{2}} D_{2} \left( \lVert F \rVert_{\calL^{2}(O_{3}, \sigma^{*}\bar{g})} + \lVert u \rVert_{\calL^{2}(O_{3}, \sigma^{*}\bar{g})} \right) \\
        & \leqslant \epsilon^{2 - \frac{n}{2}}D_{2}\lVert F \rVert_{\calL^{2}(O_{3}, \sigma^{*}\bar{g})} + \epsilon^{2 - \frac{n}{2}}D_{1} D_{2}\lVert u \rVert_{\calL^{\frac{2(n + 2)}{n}}(O_{3}, \sigma^{*}\bar{g})} \\
        & \leqslant \epsilon^{2} D_{2}  \lVert F \rVert_{\calL^{2}(O_{3}, \sigma^{*}(\epsilon^{-2}\bar{g}))} + \epsilon^{2 - \frac{n}{n + 2}} D_{1}D_{2} \lVert u \rVert_{\calL^{\frac{2(n + 2)}{n}}(O_{3}, \sigma^{*} (\epsilon^{-2} \bar{g}))} \\
        & \leqslant \epsilon^{2} D_{2}  \lVert F \rVert_{\calL^{2}(W, \sigma^{*}(\epsilon^{-2}\bar{g}))} + \epsilon^{2 - \frac{n}{n + 2}} D_{1}D_{2} \lVert u \rVert_{\calL^{\frac{2(n + 2)}{n}}(W, \sigma^{*} (\epsilon^{-2} \bar{g}))}.
    \end{split}
\end{equation}
By Sobolev embedding inequality with respect to $ \sigma^{*} (\epsilon^{-2} \bar{g}) $,
\begin{equation}\label{Pre:eqn15}
\left\lVert \frac{\partial^{2}u}{\partial (t')^{2}} \right\rVert_{\calC^{0}(O_{1})} \leqslant D_{0} \epsilon^{\frac{n}{2}} \left\lVert \frac{\partial^{2}u}{\partial (t')^{2}} \right\rVert_{H^{s}(O_{1}, \sigma^{*}(\epsilon^{-2} \bar{g}))}.
\end{equation}
Here $ D_{0}, D_{1}, D_{2} $ are independent of $ \epsilon $ and $ u $. It is well-known that the Yamabe constant (\ref{Pre:eqn10}) is invariant under conformal transformation, so is invariant under the scaling of the metrics. By (\ref{Pre:eqn10}) and Lemma \ref{Pre:lemma4}, $ Y(M \times \mathbb{S}_{t}^{1}, \bar{g}) $ is finite and bounded below by some positive constant since $ \bar{g} $ has uniformly positive scalar curvature. Denote by $ \lambda : = Y(M \times \mathbb{S}_{t}^{1}, \bar{g})$, and recall that $ \dim (M \times \mathbb{S}^{1}) = n + 2 $, we have
\begin{equation*} \lVert v \rVert_{\calL^{\frac{2(n +2)}{n}}(M \times \mathbb{S}^{1}, \epsilon^{-2}\bar{g})}^{2} \leqslant \lambda^{-1} \left( \frac{4(n + 1)}{n} \lVert \nabla_{\epsilon^{-2}\bar{g}} v \rVert_{\calL^{2}(M \times \mathbb{S}^{1}, \epsilon^{-2}\bar{g})}^{2} + \int_{M \times \mathbb{S}^{1}} R_{\epsilon^{-2}\bar{g}} v^{2} d\text{Vol}_{\epsilon^{-2}\bar{g}} \right). 
\end{equation*}
that holds for every $ v \in \calC_{c}^{\infty}(M \times \mathbb{S}^{1}) \backslash \lbrace 0 \rbrace $. Denote by $ X_{\xi, \zeta} : = X \times \lbrace \xi \rbrace \times \lbrace \zeta \rbrace \subset M $. For every $ \xi \in \R_{\xi}, \zeta \in \R_{\zeta} $, we denote the natural inclusion by $ \sigma_{\xi, \zeta} : X_{\xi, \zeta} \times \mathbb{S}^{1} \rightarrow M \times \mathbb{S}^{1} $, and the space $ X_{\xi, \zeta} \times \mathbb{S}^{1} $ by $ W_{\xi, \zeta} $. Since $ M $ is of bounded curvature and $ \bar{g} $ is a product metric, $ (M \times \mathbb{S}^{1}, \bar{g}) $ is of bounded curvature. By \cite[Lemma 2.25, Lemma 2.26]{Aubin} which are essentially due to Calabi, there exist two positive constants $ D_{3} $ and $ D_{4} $ such that two metrics $ \bar{g} $ and $ \sigma_{\xi, \zeta}^{*} \bar{g} \oplus d\xi^{2} \oplus d\zeta^{2} $ satisfy the pointwise relation
\begin{equation*}
    D_{4}^{-2} \sqrt{\det{\bar{g}}} \leqslant \sqrt{\det(\sigma_{\xi}^{*} \bar{g} \oplus d\xi^{2} \oplus d\zeta^{2})} \leqslant D_{3}^{\frac{2(n + 2)}{n}} \sqrt{\det{\bar{g}}}, \forall \xi, \zeta.
\end{equation*}
Set $ \xi' = \epsilon^{-1} \xi, \zeta' = \epsilon^{-1} \zeta $, we have
\begin{equation*}
   D_{4}^{-2} \sqrt{\det{(\epsilon^{-2}\bar{g})}} \leqslant \sqrt{\det(\sigma_{\xi}^{*}(\epsilon^{-2} \bar{g}) \oplus d(\xi')^{2} \oplus d(\zeta')^{2})} \leqslant D_{3}^{\frac{2(n + 2)}{n}} \sqrt{\det{(\epsilon^{-2}\bar{g})}}, \forall \xi', \zeta'.
\end{equation*}
Applying the the positive lower bound of the Yamabe quotient, the volume form comparison, and note the construction of $ \Phi $ and $ \Phi' $ in (\ref{Pre:eqn9}), we have
\begin{equation}\label{Pre:eqn16}
\begin{split}
& \lVert u \rVert_{\calL^{\frac{2(n + 2)}{n}}(W, \sigma^{*} (\epsilon^{-2} \bar{g}))} \\
& \qquad = (2\epsilon)^{\frac{n}{n + 2}} \left( \int_{-\epsilon^{-1}}^{\epsilon^{-1}} \int_{-\epsilon^{-1}}^{\epsilon^{-1}} \int_{W} \lvert u_{1} \rvert^{\frac{2(n + 1)}{n - 1}} d\text{Vol}_{\sigma^{*} (\epsilon^{-2} \bar{g})} d\xi'd\zeta'  \right)^{\frac{n}{2(n + 2)}} \\
& \qquad = (2\epsilon)^{\frac{n}{n + 2}} \left( \int_{W \times [-\epsilon^{-1}, \epsilon^{-1}]^{2}} \lvert u_{1} \rvert^{\frac{2(n + 2)}{n}} d\text{Vol}_{\sigma^{*}(\epsilon^{-2} \bar{g}) \oplus d(\xi')^{2}}  \right)^{\frac{n}{2(n + 2)}} \\
& \qquad = (2\epsilon)^{\frac{n}{n + 2}} \left( \int_{W \times [-\epsilon^{-1}, \epsilon^{-1}]^{2}} \lvert \Phi(\cdot, \xi', \zeta')\Phi'(\cdot, \xi', \zeta') \cdot u_{1} \rvert^{\frac{2(n + 2)}{n}} d\text{Vol}_{\sigma^{*} (\epsilon^{-2} \bar{g}) \oplus d(\xi')^{2}}  \right)^{\frac{n }{2(n + 2)}}  \\
& \qquad \leqslant (2\epsilon)^{\frac{n}{n + 2}} D_{3}  \left( \int_{M \times \mathbb{S}^{1}_{t'}}  \lvert \bar{u} \rvert^{\frac{2(n + 2)}{n}} d\text{Vol}_{\epsilon^{-2}\bar{g}} \right)^{\frac{n}{2(n + 2)}} \\
& \qquad \leqslant (2\epsilon)^{\frac{n}{n + 2} } D_{3} \lambda^{-\frac{1}{2}} \left( \frac{4(n + 1)}{n} \lVert \nabla_{\epsilon^{-2}\bar{g}} \bar{u} \rVert_{\calL^{2}(M \times \mathbb{S}^{1}_{t'}, \epsilon^{-2}\bar{g})} + \int_{M \times \mathbb{S}^{1}_{t'}} R_{\epsilon^{-2}\bar{g}} \bar{u}^{2} d\text{Vol}_{\epsilon^{-2}\bar{g}} \right)^{\frac{1}{2}} \\
& \qquad \leqslant (2\epsilon)^{\frac{n}{n + 2}} D_{3} D_{4} \lambda^{-\frac{1}{2}} \left( \frac{4(n + 1)}{n}\int_{\R} \int_{\R} \int_{W_{\xi, \zeta}} \lvert \nabla_{\sigma_{\xi}^{*}(\epsilon^{-2}\bar{g})} \bar{u} \rvert^{2} d\text{Vol}_{\sigma_{\xi}^{*}(\epsilon^{-2}\bar{g})} d\xi' d\zeta' \right)^{\frac{1}{2}} \\
& \qquad \qquad + (2\epsilon)^{\frac{n}{n + 2}} D_{3} D_{4} \lambda^{-\frac{1}{2}} \left( \int_{\R} \int_{\R} \int_{W_{\xi, \zeta}} R_{\epsilon^{-2}\bar{g}} \bar{u}^{2} d\text{Vol}_{\sigma^{*} (\epsilon^{-2}\bar{g})} d\xi'd\zeta' \right)^{\frac{1}{2}} \\
& = 2^{\frac{n}{n + 2}} \epsilon^{-1} D_{3} D_{4} \lambda^{-\frac{1}{2}} \left( \frac{4(n + 1)}{n}\int_{\R} \int_{\R} \int_{W_{\xi, \zeta}} \lvert \nabla_{\sigma_{\xi, \zeta}^{*}(\epsilon^{-2}\bar{g})} (\Phi(\cdot, \xi, \zeta) \Phi'(\cdot, \xi, \zeta) u_{1}) \rvert^{2} d\text{Vol}_{\sigma_{\xi, \zeta}^{*}(\epsilon^{-2}\bar{g})} d\xi \right)^{\frac{1}{2}} \\
& \qquad \qquad + 2^{\frac{n}{n + 2}} \epsilon^{-1} D_{3} D_{4} \lambda^{-\frac{1}{2}} \left( \int_{\R} \int_{\R} \int_{W_{\xi, \zeta}} R_{\epsilon^{-2}\bar{g}} (\Phi(\cdot, \xi, \zeta)\Phi'(\cdot, \xi, \zeta) u_{1})^{2} d\text{Vol}_{\sigma^{*} (\epsilon^{-2}\bar{g})} d\xi \right)^{\frac{1}{2}} \\
& \qquad \leqslant 2D_{3} D_{4}D_{\Phi, \Phi'} \lambda^{-\frac{1}{2}} \epsilon^{-\frac{2}{n + 2}} \times \\
& \qquad \qquad \times \left( \frac{4(n + 1)}{n} \max_{\xi \in I_{\xi}, \zeta \in I_{\zeta}} \left( \lVert \nabla_{\sigma_{\xi, \zeta}^{*} (\epsilon^{-2} \bar{g})}u_{1} \rVert_{\calL^{2}(W_{\xi}, \sigma_{\xi, \zeta}^{*}(\epsilon^{-2}\bar{g}))}^{2} \right) + \epsilon^{2} \max_{W} \lvert R_{\bar{g}} \rvert \lVert u \rVert_{\calL^{2}(W, \sigma^{*}(\epsilon^{-2}\bar{g}))}^{2} \right)^{\frac{1}{2}}
\end{split}
\end{equation}
Here $ D_{\Phi, \Phi'} $ depends on $ \Phi, \Phi' $ and the supports $ I_{\xi}, I_{\zeta} $, and is independent of $ \epsilon $ and $ u $. We absorb the constant $ 2^{\frac{n}{n + 2}} $ into $ D_{\Phi, \Phi'} $ for simplification. In this long derivation, we must use the noncompact version of the Yamabe constant due to the support of $ \Phi \cdot \Phi' $ with respect to $ \xi', \zeta' $ variables.

Note that there are trivial diffeomorphisms between $ W $ and $ W_{\xi, \zeta} $ for every $ \xi \in I_{\xi}, \zeta \in I_{\zeta} $. Note that $ \bar{u} $ is constant for both $ \xi $-variable and $ \zeta $-variable. By compactness of $ W \times I_{\xi} \times I_{\zeta} $ and smallness of $ \calC^{1, \alpha} $-norm of $ u $, we have $ \lVert \nabla_{\sigma_{\xi, \zeta}^{*}\bar{g}} \bar{u} \rVert_{\calL^{2}(W_{\xi, \zeta}, \sigma_{\xi, \zeta}^{*} \bar{g})} \leqslant D_{5} \eta $ for some $ D_{5} $ independent of $ \epsilon $. Without loss of generality, we may assume $ \max_{W} \lvert R_{\bar{g}} \rvert \leqslant D_{5} $ by increasing $ D_{5} $ if necessary. Note also that by Lemma \ref{Pre:lemma2}, $ \lVert F \rVert_{\calL^{p}(W, \sigma^{*}\bar{g})} $ is invariant under the scaling of the size in $ t $-direction since in a tubular neighborhood of $ X \times \lbrace 0 \rbrace $, we write $ F = (C + 1) \cdot (\Pi_{2})^{*} \phi(t) $ by \cite[Lemma 2.2]{XU11}, hence is small in the sense that $ (C + 1)^{p} \epsilon \ll 1 $, hence $ (C + 1)^{2} \epsilon \ll 1 $ since $ p > n \geqslant 3 $. Combining (\ref{Pre:eqn6}), (\ref{Pre:eqn13}), (\ref{Pre:eqn14}), (\ref{Pre:eqn15}), (\ref{Pre:eqn16}), we have
\begin{align*}
& \left\lVert \frac{\partial^{2}u}{\partial (t')^{2}} \right\rVert_{\calC^{0}(O_{1})} \\ 
& \qquad \leqslant D_{0}\epsilon^{\frac{n}{2}} \lVert \frac{\partial^{2}u}{\partial (t')^{2}} \rVert_{H^{s}(O_{1}, \sigma^{*}(\epsilon^{-2} \bar{g}))} \leqslant D_{0} D_{s} \epsilon^{\frac{n}{2}} \left\lVert \frac{\partial^{2}u}{\partial (t')^{2}} \right\rVert_{\calL^{2}(O_{2}, \sigma^{*}(\epsilon^{-2} \bar{g}))} \\
& \qquad \leqslant \epsilon^{2 + \frac{n}{2}} D_{0}D_{2}D_{s}  \lVert F \rVert_{\calL^{2}(W, \sigma^{*}(\epsilon^{-2}g))} + \epsilon^{2 - \frac{n}{n + 2} + \frac{n}{2}} D_{0}D_{1}D_{2}D_{s} \lVert u \rVert_{\calL^{\frac{2(n + 1)}{n - 1}}(W, \sigma^{*} (\epsilon^{-2} \bar{g}))} \\
& \qquad \leqslant 2D_{0} D_{1} D_{2} D_{3}D_{4} D_{s} D_{\Phi} \lambda^{-\frac{1}{2}}  \epsilon^{1 + \frac{n}{2}} \times \\
& \qquad \qquad \times \left( \frac{4(n + 1)}{n} \max_{\xi \in I_{\xi}, \zeta \in I_{\zeta}} \left( \lVert \nabla_{\sigma_{\xi, \zeta}^{*} (\epsilon^{-2} \bar{g})} u_{1} \rVert_{\calL^{2}(W_{\xi, \zeta}, \sigma_{\xi, \zeta}^{*}(\epsilon^{-2}\bar{g}))}^{2} \right) + \epsilon^{2} \max_{W} \lvert R_{\bar{g}} \rvert \lVert u \rVert_{\calL^{2}(W, \sigma^{*}(\epsilon^{-2}\bar{g}))}^{2} \right)^{\frac{1}{2}} \\
& \qquad \qquad \qquad + \epsilon^{2 + \frac{n}{2}} D_{0}D_{2}D_{s}  \lVert F \rVert_{\calL^{2}(W, \sigma^{*}(\epsilon^{-2}\bar{g}))} \\
& \qquad : = \bar{D}_{0}  \epsilon^{2} \left( \frac{4(n + 1)}{n} \max_{\xi \in I_{\xi}, \zeta \in I_{\zeta}} \left( \lVert \nabla_{\sigma_{\xi, \zeta}^{*}  \bar{g}} u_{1} \rVert_{\calL^{2}(W_{\xi, \zeta}, \sigma_{\xi, \zeta}^{*}\bar{g})}^{2} \right) + \max_{W} \lvert R_{\bar{g}} \rvert \lVert u \rVert_{\calL^{2}(W, \sigma^{*}\bar{g})}^{2} \right)^{\frac{1}{2}} \\
& \qquad + \epsilon^{2 + \frac{n}{2}} D_{0}D_{2}D_{s}  \lVert F \rVert_{\calL^{2}(W, \sigma^{*}(\epsilon^{-2}\bar{g}))} \\
& \qquad < \frac{8(n + 1)}{n} \bar{D}_{0} D_{5} \epsilon^{2} \eta + D_{0}D_{2}D_{s} \epsilon^{2} \text{Vol}_{g}(X) (C + 1) \epsilon^{\frac{1}{2}} \\
& \qquad : = \bar{D} \epsilon^{2} \eta + \bar{D}' \epsilon^{2 + \frac{1}{2}} (C + 1).
\end{align*}
By taking smaller $ \epsilon $ that resulting in smaller $ \eta $ in Lemma \ref{Pre:lemma3} if necessary, it follows that
\begin{equation*}
\left\lVert \frac{\partial^{2}u}{\partial t^{2}} \right\rVert_{\calC^{0}(U_{1})} = \epsilon^{-2} \left\lVert \frac{\partial^{2}u}{\partial (t')^{2}} \right\rVert_{\calC^{0}(O_{1})} < \bar{D} \eta + \bar{D}'(C + 1) \epsilon^{\frac{1}{2}} : = \eta' \ll 1.
\end{equation*}
As desired.
\end{proof}
\begin{remark}\label{Pre:re3}
Since $ u_{W} = u + 1 $ and $ \Phi \cdot \Phi' \equiv 1 $ in $ X \times [-1, 1]_{\xi} \times [-1, 1]_{\zeta} \times \times \mathbb{S}_{t}^{1} $, Lemma \ref{Pre:lemma5} and (\ref{Pre:eqn11}) implies that
\begin{equation}\label{Pre:eqn17}
    \left\lVert \frac{\partial^{2} u_{W}}{\partial t^{2}} \right\rVert_{\calC^{0}(X \times \lbrace 0 \rbrace_{\xi} \times \lbrace 0 \rbrace_{\zeta} \times \lbrace P \rbrace_{t})} < \eta' \Rightarrow \left\lvert \frac{\partial^{2} \hat{u}}{\partial t^{2}} \bigg|_{X \times \lbrace 0 \rbrace_{\xi} \times \lbrace 0 \rbrace_{\zeta} \times \left(- \frac{\epsilon}{4}, \frac{\epsilon}{4} \right)_{t}} \right\rvert < \eta'.
    \end{equation}
\end{remark}
By (\ref{intro:eqn3}), we need to apply $ \Delta_{g} u_{M} $ and $ \Delta_{\imath_{\xi}^{*}g}u_{Y} $, but we only have a differential relation in terms of $ \Delta_{\sigma^{*}g} u_{W} $. We need to compute the relations among them via the Laplacian $ \Delta_{\bar{g}} \hat{u} $ in the ambient space $ M \times \mathbb{S}^{1} $.
\begin{lemma}\label{Pre:lemma6}
Under the hypotheses of Lemma \ref{Pre:lemma5},
\begin{equation}\label{Pre:eqn18}
\begin{split}
\Delta_{\bar{g}} \hat{u} & = \Delta_{g} u_{M} + \frac{\partial^{2} \hat{u}}{\partial t^{2}} \; {\rm on} \; X \times \R_{\xi} \times \R_{\zeta} \times \lbrace P \rbrace_{t}, \\
\Delta_{\bar{g}} \hat{u} & = \Delta_{\imath_{\xi}^{*}g} u_{Y} + \frac{\partial^{2} \hat{u}}{\partial t^{2}} \; {\rm on} \; X \times \R_{\zeta} \times \lbrace 0 \rbrace_{\xi} \times \lbrace P \rbrace_{t}, \\
\Delta_{\bar{g}}\hat{u} & = \Delta_{\sigma^{*}\bar{g}} u_{W} + A_{1}(\bar{g}, M, u) \; {\rm on} \; X \times \lbrace 0 \rbrace_{\xi} \times \lbrace 0 \rbrace_{\zeta} \times \mathbb{S}_{t}^{1},
\end{split}
\end{equation}
where $ A_{1}(g, M, u) $ involves zeroth and first order derivatives of $ u $.
\end{lemma}
\begin{proof}
Let's verify the last equality of (\ref{Pre:eqn18}) first. By the diagrams (\ref{Pre:eqn6}), it is equivalent to compute
\begin{equation*}
    \Delta_{\bar{g}}\hat{u} - \Delta_{\sigma^{*}\bar{g}} u_{W} = \Delta_{\bar{g}}\hat{u} - \Delta_{\imath_{\xi}^{*}g \oplus dt^{2}} (\pi')^{*}u_{W} + \Delta_{\imath_{\xi}^{*}g \oplus dt^{2}} (\pi')^{*}u_{W} - \Delta_{\sigma^{*}\bar{g}} u_{W}
\end{equation*}
By (\ref{Pre:eqn7}) and (\ref{Pre:eqn8}), it follows that $ (\pi')^{*}u_{W} $ is constant with $ \zeta $-variable, and $ \hat{u} $ is constant with both $ \xi $-variable and $ \zeta $-variable. Therefore by \cite[Lemma 2.5]{XU11}, there are $ B_{1}(\bar{g}, M, u) $ and $ B_{2}(\bar{g}, M, u) $, which consists of zeroth and first orders of derivatives of $ u $ only, such that
\begin{align*}
    \Delta_{\bar{g}}\hat{u} - \Delta_{\imath_{\xi}^{*}g \oplus dt^{2}} (\pi')^{*}u_{W}  & = B_{1}(g, M, u), \\
    \Delta_{\imath_{\xi}^{*}g \oplus dt^{2}} (\pi')^{*}u_{W} - \Delta_{\sigma^{*}\bar{g}} u_{W} & = B_{2}(g, M, u).
\end{align*}
Denote by $ A_{1}(g, M, u) : = B_{1}(g, M, u) + B_{2}(g, M, u) $, the last equality of (\ref{Pre:eqn18}) follows.

Choose any local chart of $ M $, that is of the form $ (x^{1}, \dotso, x^{n - 1}, \zeta, \xi, t) $. Since $ \bar{g} = g \oplus dt^{2} $, $ \hat{u} $ is constant along each $ \xi $-fiber and $ \zeta $-fiber, $ u_{Y} $ is cosntant along each $ \zeta $-fiber, the first two equalities of (\ref{Pre:eqn18}) follows.
\end{proof}
\medskip

\section{Riemannian Metric with Positive Scalar Curvature on $ X $}
In this section, we show that for any complete metric $ g $ on the noncompact space $ M \times \R^{2} $ such that (i) $ g $ is of bounded curvature, (ii) $ R_{g} \geqslant \kappa_{0} > 0 $  for some $ \kappa_{0} \in R_{+} $ and (iii) satisfying the $ g $-angle condition (\ref{intro:eqn1a}) or (\ref{intro:eqn1b}) on $ X_{\zeta, 0} $, there exists a conformal metric $ \tilde{g} \in [g] $ such that $ R_{\imath_{\zeta}^{*}\imath_{\xi}^{*} \tilde{g}} > 0 $ on $ X $. This completes the Step I of our Riemannian path, and partially answers the 1994 Rosenberg-Stolz conjecture \cite[Conjecture 7.1(2)]{RosSto} in all dimensions.

As a final preparation, we denote by 
\begin{align*}
    & e^{2\hat{\phi}} : = (\hat{u})^{\frac{4}{n - 2}} \Rightarrow \phi_{W} : = \sigma^{*} \hat{\phi},  \phi_{M} : = \tau_{1}^{*} \hat{\phi}, \phi_{Y} : = \imath_{\xi}^{*} \tau_{1}^{*} \hat{\phi} \\
    \Rightarrow & e^{2\phi_{W}} = (u_{W})^{\frac{4}{n - 2}}, e^{2\phi_{M}} = (u_{M})^{\frac{4}{n - 2}}, e^{2\phi_{Y}} = (u_{Y})^{\frac{4}{n-2}}.
\end{align*}
For any two functions $ \varphi, v $ with the relation $ e^{2\varphi} = v^{\frac{4}{n - 2}}, n \geqslant 3 $, any Riemannian metric $ g_{0} $ and any global vector $ V $ in appropriate spaces, we have
\begin{equation}\label{Riem:eqn0}
\begin{split}
    & v^{-\frac{n+2}{n- 2}}\left(-\frac{4}{n - 2} \Delta_{g_{0}} v 
    \right) = e^{-2\varphi} \left(  
    - 2 \Delta_{g_{0}} \varphi - (n - 2) \lvert \nabla_{g_{0}} \varphi \rvert^{2} \right), \\
   &  e^{-2\varphi} \nabla_{g_{0}} \varphi = \frac{2}{n - 2} v^{-\frac{n+2}{n-2}} \nabla_{g_{0}} v,\  e^{-2\varphi} \lvert \nabla_{g_{0}} \varphi \rvert^{2} = \left( \frac{2}{n - 2} \right)^{2}v^{-\frac{2n}{n - 2}} \lvert \nabla_{g_{0}} v \rvert^{2}, \\
   & e^{-2\varphi} \left( 2(n - 2) \nabla_{V} \nabla_{V} \varphi + (n - 2)^{2} \nabla_{V} \varphi \nabla_{V} \varphi \right)   
= 4v^{-\frac{n + 2}{n - 2}}  \nabla_{V} \nabla_{V} v.
\end{split}
\end{equation}  
We are now ready to prove our main theorem in Riemannian geometry. Note that we only deal with the hard case where $ (d\pi_{\xi})_{*}\nu_{g} $ is nowhere vanishing on $ X_{\xi, 0} $. When $ (d\pi_{\xi})_{*}\nu_{g} \equiv 0 $, the following theorem holds due to \cite[Theorem 3.1]{XU11} with the degenerated angle condition given in Remark \ref{Pre:re1}.
\begin{theorem}\label{Riem:thm1}
Let $ X $ be an oriented, closed manifold with $ n - 1 = \dim_{\R}X \geqslant 2 $. Assume that $ X \times \R^{2} $ admits a Riemannian metric $ g $ that is of bounded curvature, and such that $ R_{g} \geqslant \kappa_{0} > 0 $ for some $ \kappa_{0} > 0 $. Assume
\begin{equation}\label{Riem:eqn1a}
\begin{split}
 & \frac{n - 1}{n - 2} \sec^{2} (\angle_{\imath_{\xi}^{*}g}(d\pi_{\xi})_{*} \nu_{g}, \partial_{\zeta})) + \sec^{2} (\angle_{\imath_{\xi}^{*}g}(\nu_{\imath_{\xi}^{*}g}, \partial_{\zeta}))  \\
 & \qquad = \frac{n - 1}{n - 2} \frac{\imath_{\xi}^{*}g(\partial_{\zeta},\partial_{\zeta})}{\imath_{\xi}^{*}g((d\pi_{\xi})_{*}\nu_{g},\partial_{\zeta})^{2}} + \frac{\imath_{\xi}^{*}g(\partial_{\zeta},\partial_{\zeta})}{\imath_{\xi}^{*}g(\nu_{\imath_{\xi}^{*}g},\partial_{\zeta})^{2}}<
 2 + \frac{n - 1}{n - 2}
\end{split}
\end{equation}
on $ X \times \lbrace 0 \rbrace_{\xi} \times \lbrace 0 \rbrace_{\zeta} $ if $ (d\pi_{\xi})_{*}\nu_{g} $ is nowhere vanishing along $ X \times \lbrace 0 \rbrace_{\xi} \times \lbrace 0 \rbrace_{\zeta} $; otherwise assume
\begin{equation}\label{Riem:eqn1b}
\sec^{2} (\angle_{\imath_{\xi}^{*}g}(\nu_{\imath_{\xi}^{*}g}, \partial_{\zeta}))   = \frac{\imath_{\xi}^{*}g(\partial_{\zeta},\partial_{\zeta})}{\imath_{\xi}^{*}g(\nu_{\imath_{\xi}^{*}g},\partial_{\zeta})^{2}} < 2 
\end{equation}
when $ (d\pi_{\xi})_{*}\nu_{g} \equiv 0 $ on $ X \times \lbrace 0 \rbrace_{\xi} \times \lbrace 0 \rbrace_{\zeta} $, Then there exists a complete Riemannian metric $ \tilde{g} $ in the conformal class of $ g $ such that $ R_{\imath_{\zeta}^{*} \imath_{\xi}^{*} \tilde{g}} > 0 $ on $ X $.
\end{theorem}
\begin{proof}
If $ d(\pi_{\xi})_{*}\nu_{g} \equiv 0 $ on $ X_{\xi, 0} $, Theorem 3.1 of \cite{XU11} applies with the angle condition (\ref{Riem:eqn1b}) and we are done. 

Assume that $ d(\pi_{\xi})_{*}\nu_{g} $ never vanishes on $ X_{\xi, 0} $. We set $ p, \alpha $ as in Lemma \ref{Pre:lemma3}. Choose $ C > 0 $ such that
\begin{equation}\label{Riem:eqn0a}
C > 2 \max_{X \times \lbrace 0 \rbrace_{\xi} \times \lbrace 0 \rbrace_{\zeta} \times \lbrace P \rbrace_{t}} \left( 2 \lvert {\rm Ric}_{g}(\nu_{g}, \nu_{g}) \rvert + 2 \lvert {\rm Ric}_{\imath_{\xi}^{*}g}(\nu_{\imath_{\xi}^{*}g}, \nu_{\imath_{\xi}^{*}g}) \rvert + h_{g}^{2} + \lvert A_{g} \rvert^{2} + h_{\imath_{\xi}^{*}g}^{2} + \lvert A_{\imath_{\xi}^{*}g} \rvert^{2} \right)  + 3.
\end{equation}
We choose $ \eta, \eta' \ll 1 $, and associated small enough $ \delta $, $ \epsilon $, and finally associated $ F $ in Lemma \ref{Pre:lemma2}, such that:

(i) the solution of (\ref{Pre:eqn5}) satisfies (\ref{Pre:eqn6}) and (\ref{Pre:eqn11}); 

(ii) $ 4\lvert A_{1}(g, M, u) \rvert < 1 $ on $ X \times \lbrace 0 \rbrace_{\xi} \times \lbrace 0 \rbrace_{\zeta} \times \lbrace P \rbrace_{t} $; 

(iii) The following two inequalities holds
\begin{equation}\label{Riem:eqn0b}
\begin{split}
   \frac{1}{2} < \lVert u_{W} \rVert_{\calC^{0}(X \times [-\frac{\epsilon}{2}, \frac{\epsilon}{2}]_{t})} & < \frac{3}{2}; \\
   \frac{4}{(n - 2)^{2}} \left\lvert  n(n- 1) \frac{\left\lvert \nabla_{V_{1}} u_{W} \right\rvert^{2}}{u_{W}} + n(n - 2)\frac{\left\lvert \nabla_{V_{2}} u_{W} \right\rvert^{2}}{u_{W}} +(n - 2) \frac{\left\lvert \nabla_{\imath_{\xi}^{2}g} u_{W} \right\rvert^{2}}{u_{W}}  \right\rvert & < 1 \; {\rm on} \; X \cong X_{\zeta, 0}.
\end{split}
\end{equation}
By (\ref{Riem:eqn0b}) and (\ref{Pre:eqn7}), it follows that 
\begin{equation*}
u_{M} \in \left[\frac{1}{2}, \frac{3}{2} \right] \; {\rm on} \; M.    
\end{equation*} 
With the conformal factor $ (u_{M})^{\frac{4}{n - 2}} $, we define
\begin{equation*}
    \tilde{g} = e^{2\phi_{M}} g = (u_{M})^{\frac{4}{n - 2}} g.
\end{equation*}
$ \tilde{g} $ is a complete metric since $ u_{M} $ is uniformly bounded above and below, and $ g $ is complete by assumption.

Apply the Gauss-Codazzi equations (\ref{intro:eqn3}) on $ X_{\zeta, 0} = X \times \lbrace 0 \rbrace_{\xi} \times \lbrace 0 \rbrace_{\zeta} $ with respect to the metric $ \tilde{g} $ in $ M $ and $ \imath_{\xi}^{*} \tilde{g} $ in $ Y = X \times \R_{\zeta} $, and note that $ \nu_{\tilde{g}} = e^{-\phi_{M}} \nu_{g} $, $ \nu_{\imath_{\xi}^{*} \tilde{g}} = e^{-\phi_{Y}} \nu_{\imath_{\xi}^{*}g} $, we have
\begin{equation}\label{Riem:eqn2}
\begin{split}
R_{\imath_{\zeta}^{*} \imath_{\xi}^{*}\tilde{g}} & = \left( R_{\tilde{g}} - 2\text{Ric}_{\imath_{\xi}^{*}\tilde{g}}(\nu_{\imath_{\xi}^{*}\tilde{g}}, \nu_{\imath_{\xi}^{*}\tilde{g}}) - 2\text{Ric}_{\tilde{g}}(\nu_{\tilde{g}}, \nu_{\tilde{g}}) + h_{\tilde{g}}^{2} + h_{\imath_{\xi}^{*}\tilde{g}}^{2} - \lvert A_{\tilde{g}} \rvert^{2} - \lvert A_{\imath_{\xi}^{*}\tilde{g}} \rvert^{2} \right) \bigg|_{X_{\zeta, 0}} \\
& = e^{-2\phi_{X}} \left( R_{g} - 2\text{Ric}_{\imath_{\xi}^{*}g}(\nu_{\imath_{\xi}^{*}g}, \nu_{\imath_{\xi}^{*}g}) - 2\text{Ric}_{g}(\nu_{g}, \nu_{g}) + h_{g}^{2} + h_{\imath_{\xi}^{*}g}^{2} - \lvert A_{g} \rvert^{2} - \lvert A_{\imath_{\xi}^{*}g} \rvert^{2} \right) \bigg|_{X_{\zeta, 0}} \\
& \qquad + e^{-2\phi_{X}} \left(2(n - 1)\nabla_{\nu_{g}} \nabla_{\nu_{g}} \phi_{M} + 2(n - 2) \nabla_{\nu_{\imath_{\xi}^{*}g}} \nabla_{\nu_{\imath_{\xi}^{*}g}} \phi_{Y} -2(n - 1) \Delta_{g} \phi_{M} + 2 \Delta_{\imath_{\xi}^{*}g} \phi_{Y} \right) \bigg|_{X_{\zeta, 0}} \\
& \qquad \qquad + e^{-2\phi_{X}} \left( - (n - 2)(n - 1) \lvert \nabla_{g} \phi_{M} \rvert_{g}^{2} + 2(n - 2) \lvert \nabla_{\imath_{\xi}^{*}g} \phi_{Y} \rvert_{\imath_{\xi}^{*}g}^{2} \right) \bigg|_{X_{\zeta, 0}} \\
& \qquad \qquad \qquad + e^{-2\phi_{X}} \left( - 2(n - 1) \left( \nabla_{\nu_{g}} \phi_{Y} \right)^{2} - 2(n - 2) \left( \nabla_{\nu_{\imath_{\xi}^{*}g}} \phi_{Y} \right)^{2} \right) \bigg|_{X_{\zeta, 0}}.
\end{split}
\end{equation}
Note that by definition, $ \phi_{M} $ is constant with both $ \xi $-variable and $ \zeta $-variable, and $ \phi_{Y} $ is constant with $ \zeta $-variable, it follows from the definition of $ V_{1}, V_{2} $ in (\ref{Pre:eqn0}) that 
\begin{equation*}
    \left( \nabla_{\nu_{g}} \nabla_{\nu_{g}} \phi_{M} \right) |_{X_{\xi, 0}} = \left( \nabla_{V_{1}} \nabla_{V_{1}} \phi_{M} \right) |_{X_{\xi, 0}}, \left( \nabla_{\nu_{\imath_{\xi}^{*}g}}\nabla_{\nu_{\imath_{\xi}^{*}g}} \phi_{Y} \right) |_{X_{\xi, 0}} = \left( \nabla_{V_{2}} \nabla_{V_{2}} \phi_{Y} \right) |_{X_{\xi, 0}}.
\end{equation*}
Clearly the same relation holds for first order derivatives. Applying the above equalities and (\ref{Riem:eqn0}), we rewrite (\ref{Riem:eqn2}) in terms of $ u_{M}, u_{Y} $ and $ u_{X} $,
\begin{equation}\label{Riem:eqn3}
\begin{split}
R_{\imath_{\zeta}^{*} \imath_{\xi}^{*}\tilde{g}} & = (u_{X})^{-\frac{n + 2}{n - 2}} \left( R_{g} u_{M} - 2\text{Ric}_{\imath_{\xi}^{*}g}(\nu_{\imath_{\xi}^{*}g}, \nu_{\imath_{\xi}^{*}g})u_{M} - 2\text{Ric}_{g}(\nu_{g}, \nu_{g})u_{M} \right)\bigg|_{X_{\zeta, 0}} \\
& \qquad + (u_{X})^{-\frac{n + 2}{n - 2}} \left( h_{g}^{2}u_{M} + h_{\imath_{\xi}^{*}g}^{2}u_{M} - \lvert A_{g} \rvert^{2}u_{M} - \lvert A_{\imath_{\xi}^{*}g} \rvert^{2}u_{M} \right) \bigg|_{X_{\zeta, 0}}  \\
& \qquad \qquad + (u_{X})^{-\frac{n + 2}{n - 2}} \left( \frac{4(n - 1)}{n - 2} \nabla_{V_{1}} \nabla_{V_{1}} u_{M} + 4 \nabla_{V_{2}} \nabla_{V_{2}} u_{Y} \right) \bigg|_{X_{\zeta, 0}} \\
& \qquad \qquad \qquad + (u_{X})^{-\frac{n + 2}{n - 2}} \left(- \frac{4(n - 1)}{n - 2} \Delta_{g} u_{M} + \frac{4}{n - 2} \Delta_{\imath_{\xi}^{*}g} u_{Y}\right) \bigg |_{X_{\zeta, 0}} \\
& \qquad \qquad \qquad \qquad + \frac{4}{(n - 2)^{2}}(u_{X})^{-\frac{n + 2}{n - 2}} \left(-n(n- 1) \frac{\left\lvert \nabla_{V_{1}} u_{M} \right\rvert^{2}}{u_{M}} - n(n - 2)\frac{\left\lvert \nabla_{V_{2}} u_{Y} \right\rvert^{2}}{u_{Y}} \right) \bigg|_{X_{\zeta, 0}} \\
& \qquad \qquad \qquad \qquad \qquad + \frac{4}{(n - 2)^{2}}(u_{X})^{-\frac{n + 2}{n - 2}} \left( (n - 2) \frac{\left\lvert \nabla_{\imath_{\xi}^{2}g} u_{Y} \right\rvert^{2}}{u_{Y}} \right) \bigg|_{X_{\zeta, 0}}.
\end{split}
\end{equation}
Recall that $ \bar{g} = g \oplus dt^{2} $, hence $ R_{\bar{g}} = R_{g} $. Using Laplacian comparison in (\ref{Pre:eqn18}) and $ u_{M} = u_{Y} = U_{W} = u_{X} $ in $ X_{\zeta, 0} $ by (\ref{Pre:eqn8}), (\ref{Riem:eqn3}) implies that
\begin{equation}\label{Riem:eqn4}
\begin{split}
R_{\imath_{\zeta}^{*} \imath_{\xi}^{*}\tilde{g}} & = (u_{X})^{-\frac{n + 2}{n - 2}} \left( R_{g} u_{W} - 2\text{Ric}_{\imath_{\xi}^{*}g}(\nu_{\imath_{\xi}^{*}g}, \nu_{\imath_{\xi}^{*}g})u_{W} - 2\text{Ric}_{g}(\nu_{g}, \nu_{g})u_{W} \right)\bigg|_{X_{\zeta, 0}} \\
& \qquad + (u_{X})^{-\frac{n + 2}{n - 2}} \left( h_{g}^{2}u_{W} + h_{\imath_{\xi}^{*}g}^{2}u_{W} - \lvert A_{g} \rvert^{2}u_{W} - \lvert A_{\imath_{\xi}^{*}g} \rvert^{2}u_{W} \right) \bigg|_{X_{\zeta, 0}}  \\
& \qquad \qquad + (u_{X})^{-\frac{n + 2}{n - 2}} \left( \frac{4(n - 1)}{n - 2} \nabla_{V_{1}} \nabla_{V_{1}} u_{W} + 4 \nabla_{V_{2}} \nabla_{V_{2}} u_{W} \right) \bigg|_{X_{\zeta, 0}} \\
& \qquad \qquad \qquad + (u_{X})^{-\frac{n + 2}{n - 2}} \left(- \frac{4(n - 1)}{n - 2} \Delta_{\bar{g}} \hat{u} + \frac{4}{n - 2} \Delta_{\bar{g}} \hat{u} - 4 \frac{\partial^{2} \hat{u}}{\partial t^{2}} \right) \bigg |_{X_{\zeta, 0}} \\
& \qquad \qquad \qquad \qquad + \frac{4}{(n - 2)^{2}}(u_{X})^{-\frac{n + 2}{n - 2}} \left(-n(n- 1) \frac{\left\lvert \nabla_{V_{1}} u_{W} \right\rvert^{2}}{u_{W}} - n(n - 2)\frac{\left\lvert \nabla_{V_{2}} u_{W} \right\rvert^{2}}{u_{W}} \right) \bigg|_{X_{\zeta, 0}} \\
& \qquad \qquad \qquad \qquad \qquad + \frac{4}{(n - 2)^{2}}(u_{X})^{-\frac{n + 2}{n - 2}} \left( (n - 2) \frac{\left\lvert \nabla_{\imath_{\xi}^{2}g} u_{W} \right\rvert^{2}}{u_{W}} \right) \bigg|_{X_{\zeta, 0}} \\
& = (u_{X})^{-\frac{n + 2}{n - 2}} \left(- 2\text{Ric}_{\imath_{\xi}^{*}g}(\nu_{\imath_{\xi}^{*}g}, \nu_{\imath_{\xi}^{*}g})u_{W} - 2\text{Ric}_{g}(\nu_{g}, \nu_{g})u_{W} \right)\bigg|_{X_{\zeta, 0}} \\
& \qquad + (u_{X})^{-\frac{n + 2}{n - 2}} \left( h_{g}^{2}u_{W} + h_{\imath_{\xi}^{*}g}^{2}u_{W} - \lvert A_{g} \rvert^{2}u_{W} - \lvert A_{\imath_{\xi}^{*}g} \rvert^{2}u_{W} \right) \bigg|_{X_{\zeta, 0}}  \\
& \qquad \qquad + (u_{X})^{-\frac{n + 2}{n - 2}} \left( \frac{4(n - 1)}{n - 2} \nabla_{V_{1}} \nabla_{V_{1}} u_{W} + 4 \nabla_{V_{2}} \nabla_{V_{2}} u_{W} - 4\Delta_{\sigma^{*} \bar{g}} u_{W} + R_{\bar{g}} |_{W} u_{W} \right) \bigg|_{X_{\zeta, 0}} \\
& \qquad \qquad \qquad + (u_{X})^{-\frac{n + 2}{n - 2}} \left(-4A_{1}(g, M, u) - 4 \frac{\partial^{2} \hat{u}}{\partial t^{2}} \right) \bigg |_{X_{\zeta, 0}} \\
& \qquad \qquad \qquad \qquad + \frac{4}{(n - 2)^{2}}(u_{X})^{-\frac{n + 2}{n - 2}} \left(-n(n- 1) \frac{\left\lvert \nabla_{V_{1}} u_{W} \right\rvert^{2}}{u_{W}} - n(n - 2)\frac{\left\lvert \nabla_{V_{2}} u_{W} \right\rvert^{2}}{u_{W}} \right) \bigg|_{X_{\zeta, 0}} \\
& \qquad \qquad \qquad \qquad \qquad + \frac{4}{(n - 2)^{2}}(u_{X})^{-\frac{n + 2}{n - 2}} \left( (n - 2) \frac{\left\lvert \nabla_{\imath_{\xi}^{2}g} u_{W} \right\rvert^{2}}{u_{W}} \right) \bigg|_{X_{\zeta, 0}}
\end{split}
\end{equation}
By the angle condition (\ref{Riem:eqn1a}), $ u $ satisfies the elliptic partial differential equation (\ref{Pre:eqn5}). Recall that $ u_{W} = u + 1 $, it follows that
\begin{equation}\label{Riem:eqn5}
\frac{4(n - 1)}{n - 2} \nabla_{V_{1}}\nabla_{V_{1}} u_{W} + 4\nabla_{V_{2}}\nabla_{V_{2}} u_{W} - 4\Delta_{\sigma^{*}\bar{g}} u_{W} + R_{\bar{g}} |_{W} u_{W} = F + R_{\bar{g}} |_{W} \; {\rm in} \; W.
\end{equation}
Apply the smallness of $ \frac{\partial^{2} \hat{u}}{\partial t^{2}} $ on $ X_{\zeta, 0} $ in (\ref{Pre:eqn17}), the uniform bounds of $ u_{W} $ and boundedness of first order terms of $ u_{W} $ in (\ref{Riem:eqn0b}), the boundness of curvature terms in (\ref{Riem:eqn0a}), the partial differential relation in (\ref{Riem:eqn5}), and the largeness $ F = C + 1 $ on $ X_{\zeta, 0} \cong X $ in Lemma \ref{Pre:lemma2}), we estimate $ R_{\imath_{\zeta}^{*}\imath_{\xi}^{*}\tilde{g}} $ in (\ref{Riem:eqn4}) by
\begin{align*}
R_{\imath_{\zeta}^{*}\imath_{\xi}^{*}\tilde{g}} & \geqslant -(u_{X})^{-\frac{n + 2}{n - 2}} \cdot 2 \max_{X \times \lbrace 0 \rbrace_{\xi} \times \lbrace 0 \rbrace_{\zeta} \times \lbrace P \rbrace_{t}} \left( 2 \lvert {\rm Ric}_{g}(\nu_{g}, \nu_{g}) \rvert + 2 \lvert {\rm Ric}_{\imath_{\xi}^{*}g}(\nu_{\imath_{\xi}^{*}g}, \nu_{\imath_{\xi}^{*}g}) \rvert \right) \\
& \qquad - (u_{X})^{-\frac{n + 2}{n - 2}} \cdot 2\max_{X \times \lbrace 0 \rbrace_{\xi} \times \lbrace 0 \rbrace_{\zeta} \times \lbrace P \rbrace_{t}} \left( h_{g}^{2} + \lvert A_{g} \rvert^{2} + h_{\imath_{\xi}^{*}g}^{2} + \lvert A_{\imath_{\xi}^{*}g} \rvert^{2} \right) \\
& \qquad \qquad - (u_{X})^{-\frac{n + 2}{n - 2}} \cdot \frac{4}{(n - 2)^{2}} \left\lvert  n(n- 1) \frac{\left( \nabla_{V_{1}} u_{W} \right)^{2}}{u_{W}} + n(n - 2)\frac{\left( \nabla_{V_{2}} u_{W} \right)^{2}}{u_{W}} \right\rvert \\
& \qquad \qquad \qquad - (u_{X})^{-\frac{n + 2}{n - 2}} \cdot \frac{4}{(n - 2)^{2}} \left\lvert (n - 2) \frac{\left\lvert \nabla_{\imath_{\xi}^{2}g} u_{W} \right\rvert^{2}}{u_{W}}  \right\rvert \\
& \qquad \qquad \qquad \qquad + (u_{X})^{-\frac{n + 2}{n - 2}} \left( F + R_{\bar{g}} - 2 \right) \bigg|_{X_{\zeta, 0}} \\
& > 0.
\end{align*}
As desired.
\end{proof}
Let $ X_{\xi_{0}, \zeta_{0}} : = X \times \lbrace \xi_{0} \rbrace_{\xi} \times \lbrace \zeta_{0} \rbrace_{\zeta} $. Theorem \ref{Riem:thm1} also holds if we impose the same angle condition (\ref{Riem:eqn1a}) or (\ref{Riem:eqn1b}) on $ X_{\xi_{0}, \zeta_{0}} $ with associated second fundamental forms, mean curvatures, and normal vector fields. Therefore we partially answer the 1994 Rosenberg-Stolz conjecture.
\begin{corollary}\label{Riem:cor1}
Let $ X $ be an oriented, closed manifold with $ n - 1 = \dim_{\R}X \geqslant 2 $. If $ X $ admits no Reimannian metric with positive Riemannian scalar curvature, then either the angle condition (\ref{Riem:eqn1a}) or (\ref{Riem:eqn1b}) fails on all codimension two hypersurfaces $ X_{\xi_{0}, \zeta_{0}} $, the Riemannian metric $ g $ on $ X \times \R^{2} $ fails to be of bounded curvature, or there is no complete Riemannian metric $ g $ on $ X \times \R^{2} $ that admits uniformly positive Riemannian scalar curvature.
\end{corollary}
\medskip

\section{Hermitian Metric with Positive Chern Scalar Curvature on $ X \times \C $}
Let $ X $ be a closed, complex manifold with $ \dim_{\C} X \geqslant 1 $. Throughout this section, we fix an almost complex structure $ J_{0} $ on $ X $ which comes from the complex structure. We denote by $ S_{\omega} $ the Chern scalar curvature with respect to the Hermitian metric $ \omega $.

In this section, we show that the existence of Hermitian metric with positive Chern scalar curvature on noncompact manifolds of the type $ X \times \C $ by imposing some geometric conditions, provided that the complex manifold $ (X \times \C, J) $ admits a Hermitian metric $ \omega $ whose background Riemannian metric $ g $ has uniformly positive Riemannian scalar curvature. This is a generalization of X.K. Yang's result on closed Hermitian manifolds:
\begin{theorem}\label{Ch:thm1}\cite[Corollary 3.9]{Yang}
Let $ (X, J_{0}, \omega) $ be a closed Hermitian manifold such that the background Riemannian metric $ g $ has quasi-positive Riemannian scalar curvature, then there exists a Hermitian metric $ \tilde{\omega} $ such that $ S_{\tilde{\omega}} > 0 $.
\end{theorem}

\begin{theorem}\label{Ch:thm2}
Let $ X $ be a complex manifold with $ \dim_{\C}X \geqslant 1 $. Assume that $ (X \times \C, J, \omega) $ admits a complete Hermitian metric $ \omega $ whose background metric $ g $ is of bounded curvature, and such that $ R_{g} \geqslant \kappa_{0} > 0 $ for some $ \kappa_{0} > 0 $. Assume
\begin{equation}\label{Ch:eqn1a}
\begin{split}
 & \frac{n - 1}{n - 2} \sec^{2} (\angle_{\imath_{\xi}^{*}g}(d\pi_{\xi})_{*} \nu_{g}, \partial_{\zeta})) + \sec^{2} (\angle_{\imath_{\xi}^{*}g}(\nu_{\imath_{\xi}^{*}g}, \partial_{\zeta}))  \\
 & \qquad = \frac{n - 1}{n - 2} \frac{\imath_{\xi}^{*}g(\partial_{\zeta},\partial_{\zeta})}{\imath_{\xi}^{*}g((d\pi_{\xi})_{*}\nu_{g},\partial_{\zeta})^{2}} + \frac{\imath_{\xi}^{*}g(\partial_{\zeta},\partial_{\zeta})}{\imath_{\xi}^{*}g(\nu_{\imath_{\xi}^{*}g},\partial_{\zeta})^{2}}<
 2 + \frac{n - 1}{n - 2}
\end{split}
\end{equation}
on $ X \times \lbrace 0 \rbrace_{\xi} \times \lbrace 0 \rbrace_{\zeta} $ when $ (d\pi_{\xi})_{*}\nu_{g} $ is nowhere vanishing along $ X \times \lbrace 0 \rbrace_{\xi} \times \lbrace 0 \rbrace_{\zeta} $; otherwise assume
\begin{equation}\label{Ch:eqn1b}
\sec^{2} (\angle_{\imath_{\xi}^{*}g}(\nu_{\imath_{\xi}^{*}g}, \partial_{\zeta}))   = \frac{\imath_{\xi}^{*}g(\partial_{\zeta},\partial_{\zeta})}{\imath_{\xi}^{*}g(\nu_{\imath_{\xi}^{*}g},\partial_{\zeta})^{2}} < 2 
\end{equation}
when $ (d\pi_{\xi})_{*}\nu_{g} \equiv 0 $ on $ X \times \lbrace 0 \rbrace_{\xi} \times \lbrace 0 \rbrace_{\zeta} $. Then there exists a Hermitian metric $ \tilde{\omega} $ with positive Chern scalar curvature.
\end{theorem}
\begin{proof}
The fixed almost complex structure on $ X \times \C $ is $ J = J_{0} \oplus J_{1} $, where $ J_{1} $ is the complex structure of $ \C $. Since $ \omega $ is Hermitian, i.e. $ g(J\cdot, \cdot) = \omega(\cdot, \cdot) $, it follows that for $ V, W \in \Gamma(TX) $
\begin{align*}
\imath_{\zeta}^{*}\imath_{\xi}^{*}g(J_{0}V, W) & = g\left(d(\imath_{\xi} \circ \imath_{\zeta})_{*} J_{0}V, d(\imath_{\xi} \circ \imath_{\zeta})_{*} W \right) = g(Jd(\imath_{\xi} \circ \imath_{\zeta})_{*} V, d(\imath_{\xi} \circ \imath_{\zeta})_{*} W) \\
& = \omega(d(\imath_{\xi} \circ \imath_{\zeta})_{*} V, d(\imath_{\xi} \circ \imath_{\zeta})_{*} W) = \imath_{\zeta}^{*}\imath_{\xi}^{*}\omega(X, Y).
\end{align*}
It follows that the induced metric $ \imath_{\zeta}^{*}\imath_{\xi}^{*}g $ is the background Riemannian metric of the Hermitian metric $ \imath_{\zeta}^{*}\imath_{\xi}^{*}\omega $ on $ X $.

By Theorem \ref{Riem:thm1}, there exists a Riemannian metric $ \tilde{g} = e^{2\phi} \omega, \phi \in \calC^{\infty}(X \times \C) $, such that $ R_{\imath_{\zeta}^{*} \imath_{\xi}^{*} \tilde{g}} > 0 $. Clearly the conformal transformation preserves the Hermitian structure, i.e.
\begin{equation*}
    \tilde{g}(J\cdot, \cdot) = e^{2\phi}g(J\cdot, \cdot) = e^{2\phi}\omega(\cdot, \cdot) : = \tilde{\omega}(\cdot, \cdot).
\end{equation*}
It follows that $ \imath_{\zeta}^{*} \imath_{\xi}^{*} \tilde{g} $ is the background Riemannian metric of the Hermitian metric $ \imath_{\zeta}^{*} \imath_{\xi}^{*} \tilde{\omega} $. By Theorem \ref{Ch:thm1}, there exists a Hermitian metric $ \bar{\omega} $ on $ X $ such that $ S_{\bar{\omega}} > 0 $. The product Hermitian metric $ \hat{\omega} : = \bar{\omega} \oplus d\xi^{2} \oplus d\zeta^{2} $ satisfies $ S_{\hat{\omega}} > 0 $ since $ \C $ is a K\"ahler manifold with standard Euclidean metric.
\end{proof}
\medskip

\section{Generalization to $ X \times \R^{k} $ or $ X \times \C^{k} $}
In this section, we generalize our main results Theorem \ref{Riem:thm1} in Riemannian geometry and Theorem \ref{Ch:thm2} in complex geometry to the spaces $ X \times \R^{k} $ or $ X \times \C^{k} $ for any positive integer $ k $. We follow exactly the same Riemannian path, with the introduction of a generalized, conformally invariant angle condition.

For $ X \times \R^{2} $, we apply (\ref{intro:eqn3}) to get the desired positivity, i.e. by taking Gauss-Codazzi equation twice and construct a partial differential equation that is corresponding to the conformal transformation of the original metric. The angle condition is applied to obtain the ellipticity of the differential opertor $ L $ in the space $ X \times \mathbb{S}^{1} $.

Denote by the conformal metric $ \tilde{g} = e^{2\phi} g $ on $ X \times \R^{k} = X \times R_{1} \times \dotso \times \R_{k} $, and denote by $ \phi_{j - 1} = \imath_{j}^{*} \dotso \imath_{k}^{*} \phi \in \calC^{\infty}(X \times \R_{1} \times \dotso \times \R_{j - 1}), j = 2, \dotso, k $. Let $ A_{g_{j}}, h_{g_{j}} $ be the second fundamental forms and mean curvatures along $ X \times \R_{1} \times \dotso \times \R_{j - 1} $, respectively, $ j = 1, \dotso, k $. When $ k \geqslant 3 $, we apply Gauss-Codazzi $ k $ times for the metric $ \tilde{g} $, it follows that on $ X $,
\begin{equation}\label{Gen:eqn1}
\begin{split}
R_{\imath_{1}^{*} \imath_{k}^{*}\tilde{g}} 
& = e^{-2\phi_{1}} \left( R_{g} - 2\text{Ric}_{g_{1}}(\nu_{g_{1}}, \nu_{g_{ 1}}) - \dotso  - 2\text{Ric}_{g}(\nu_{g}, \nu_{g}) + h_{g}^{2} + \dotso h_{g_{1}}^{2} - \lVert A_{g} \rVert^{2} - \dotso - \lVert A_{g_{1}} \rVert^{2} \right) \\
& \qquad + e^{-2\phi_{1}} \left(2(n + k - 1)\nabla_{\nu_{g}} \nabla_{\nu_{g}} \phi + \dotso + 2(n - 2) \nabla_{g_{2}} \nabla_{g_{2}} \phi_{2} \right) \\
& \qquad \qquad  + e^{-2\phi_{1}}\left(-2(n + k - 1) \Delta_{g} \phi + 2\Delta_{g_{k - 1}} \phi_{k - 1} \dotso + 2 \Delta_{g_{1}} \phi_{1} \right) \\
& \qquad \qquad \qquad + e^{-2\phi_{1}} \left( - (n + k - 3)(n + k - 2) \lvert \nabla_{g} \phi \rvert_{g}^{2} \right) \\
& \qquad \qquad \qquad \qquad + e^{-2\phi_{1}} \left( 2(n + k - 4)  \lvert \nabla_{g_{k - 1}} \phi_{k - 1} \rvert_{g_{k - 1}}^{2} + \dotso + 2(n - 2) \lvert \nabla_{g_{2}} \phi_{2} \rvert_{g_{2}}^{2} \right) \\
& \qquad \qquad \qquad \qquad \qquad + e^{-2\phi_{1}} \left( - 2(n + k - 1) \left( \nabla_{\nu_{g}} \phi \right)^{2} - \dotso - 2(n - 2) \left( \nabla_{\nu_{g_{1}}} \phi_{1} \right)^{2} \right).
\end{split}
\end{equation}
When $ k = 2 $, (\ref{Gen:eqn1}) reduces to (\ref{intro:eqn3}); when $ k = 1 $, it further reduces to the case $ X \times \R $ and we refer to \cite{XU11}. By pairing the auxiliary space $ \mathbb{S}^{1} $, our generalization can be summerized as the following diagram:
\begin{equation}\label{Gen:diagram}
\begin{tikzcd} (X \times \R^{k}, g) \arrow[r, "\tau_{1}"] & ( X \times \R^{k} \times \mathbb{S}^{1}, \bar{g}) \arrow[d, shift left = 1.5ex, "\pi"] \\
(X, \imath_{1}^{*} \dotso \imath_{k}^{*}g) \arrow[r,"\tau_{2}"]\arrow[u,"\jmath"] &  (X \times \mathbb{S}^{1}, \sigma^{*}(\bar{g})) 
\arrow[u,"\sigma"]
\end{tikzcd}
\end{equation}
By assuming $ R_{g} \geqslant \kappa_{0} > 0 $ on $ X \times \R^{k} $, we introduce the Yamabe constant for $ X \times \R^{k} $. Due to the general Gauss-Codazzi equation (\ref{Gen:eqn1}) we only need to revise our partial differential equation
\begin{equation*}
    \frac{4(n + k - 3)}{n - 2} \nabla_{V_{1}}\nabla_{V_{1}} u + \dotso + 4 \nabla_{V_{k}}\nabla_{V_{k}} u - 4 \Delta_{\sigma^{*}\bar{g}} + R_{\bar{g}} |_{W} u = F \; {\rm in} \; W.
\end{equation*}
Here $ V_{j} $ is the projection of $ \mu_{k - j + 1} $ onto $ X $, $ j = 1, \dotso, k - 1 $, and $ V_{k} $ is the projection of $ \nu_{1} $ onto $ X $, which generalizes our choices of $ V_{1} $ and $ V_{2} $ for $ X \times \R^{2} $ case. Analogous to what we discussed in Remark \ref{Pre:re0}, we just remove the associated terms in the partial differential equation if some $ \mu_{i} \equiv 0 $ along $ X \times \R_{1} $, shall the angle condition (\ref{Gen:eqn2}) be also adjusted correspondingly as stated in the Theorem \ref{Gen:thm1} below.

Besides all necessary changes mentioned above and the generalized angle condition (\ref{Gen:eqn2}), we apply exactly the same argument as in case $ X \times \R^{2} $ to the general case $ X \times \R^{k} $, which is just notationally more difficult. It follows that the following general result is verified.
\begin{theorem}\label{Gen:thm1}
Let $ X $ be an oriented, closed manifold with $ \dim_{\R}X \geqslant 2 $. Assume that $ X \times \R^{n} $ admits a Riemannian metric $ g $ that is of bounded curvature, and such that $ R_{g} \geqslant \kappa_{0} > 0 $ for some $ \kappa_{0} > 0 $. If
\begin{equation}\label{Gen:eqn2}
\sum_{j = 2}^{n} A_{j} + \sec^{2}(\angle_{g_{1}}(\nu_{1}, \partial_{1})) < 2 + \sum_{j = 2}^{n} B_{j},
\end{equation}
along $ X \times \lbrace 0 \rbrace_{1} \times \dotso \times \lbrace 0 \rbrace_{k} $ (where, for $ j = 2, \dotso, n $, we set $ A_{j} \equiv 0 $ and remove associated $ B_{j} $ on the right hand side of (\ref{intro:eqn4}) whenever $ \mu_{j} \equiv 0 $ along $ X \times \R_{1} $), then there exists a Riemannian metric $ \tilde{g} $ in the conformal class of $ g $ such that $ R_{\imath_{\zeta}^{*} \imath_{\xi}^{*} \tilde{g}} > 0 $ on $ X $.
\end{theorem}
Here $ A_{j}, B_{j} $ are defined in (\ref{intro:eqn4a}). Again the same conclusion follows if we choose a different hypersurface, which follows that the positivity on $ X $ is obtained if the angle condition is satisfied alongs some hypersurface identified as $ X $.

When $ n = 2m $ is an even number, and $ X $ is a complex manifold, we generalize the result of Theorem \ref{intro:thm2} on $ X \times \C^{m} $, according to the same Riemannian path.
\begin{corollary}\label{Gen:cor1}
Let $ X $ be a complex manifold with $ \dim_{\C}X \geqslant 1 $. Assume that $ (X \times \C^{n}, J, \omega) $ admits a complete Hermitian metric $ \omega $ whose background metric $ g $ satisfies the same hypotheses in Theorem \ref{Gen:thm1}, then there exists a Hermitian metric $ \tilde{\omega} $ with positive Chern scalar curvature on $ X \times \C^{n} $.
\end{corollary}

\bibliographystyle{plain}
\bibliography{ScalarPre}
\end{document}